\documentclass[12pt]{article}

\usepackage[margin=1in, paperwidth=8.5in, paperheight=11in]{geometry}
\usepackage{amsfonts,amsmath,amssymb,amsthm,graphicx,lastpage,enumerate,units,color,listings,setspace}
\usepackage[super]{nth}
\usepackage[mathscr]{euscript}

\usepackage{comment}

\usepackage{hyperref}
\usepackage{theoremref}

\newcommand{\R}{\mathbb{R}}
\newcommand{\C}{\mathbb{C}}
\newcommand{\p}{\partial}

\newcommand{\vp}{\phi}
\renewcommand{\varphi}{\phi}
\newcommand{\norm}[1]{\left\lVert#1\right\rVert} 

\newcommand{\curl}{\mathop{\rm curl} \nolimits}
\newcommand{\dist}{\mathop{\rm dist}\nolimits}
\newcommand{\reals}{\R}
\newcommand{\complexes}{\C}
\renewcommand{\div}{ \mathop{\rm div}\nolimits}
\renewcommand{\widetilde}{\tilde}
\newcommand{\re}{\mathop{\rm Re}\nolimits} 
\newcommand{\im}{\mathop{\rm Im}\nolimits} 

\newenvironment{remark}[1][Remark]{\begin{trivlist}\item[\hskip \labelsep
{\it #1. }]}{  \goodbreak \end{trivlist}}  
\numberwithin{equation}{section}

\newtheorem{theorem}[equation]{Theorem}
\newtheorem{proposition}[equation]{Proposition}
\newtheorem{lemma}[equation]{Lemma}
\newtheorem{corollary}[equation]{Corollary}

\renewcommand{\theequation}{\arabic{section}.\arabic{equation}}

\begin{document}
\title{Inverse boundary value problems for polyharmonic operators with non-smooth
coefficients}
\author{R.M.~Brown\footnote{R.~Brown is partially supported by a grant from the Simons
Foundation (\#422756).}
\\ Department of Mathematics\\ University of Kentucky \\ Lexington, KY 40506 0027
\and 
L.D.~Gauthier\\Department of Mathematics \\ University of Kentucky \\ Lexington, KY 40506 0027 }
\date{}
\maketitle

\begin{abstract}
    We consider inverse boundary value problems for polyharmonic operators and in particular, 
     the problem of recovering the coefficients of terms up to order one. 
    The main
    interest of our result is that it further relaxes the regularity 
    required to establish uniqueness. The proof relies on an averaging 
    technique introduced by Haberman
    and Tataru for the study of an inverse boundary value problem 
    for a second order operator. 
\end{abstract}

\section{Introduction}\label{Intro}

The goal of this paper is to study an   inverse boundary value 
problem for a family of operators 
whose principal part is the polyharmonic operator $(-\Delta)^m$ for $m \geq 2$. 
We fix a bounded domain $ \Omega  \subset  \reals ^d$, $d\geq 3$, and 
assume the boundary of $ \Omega$ is at least Lipschitz. We consider an 
operator  $L$ given by 
\begin{equation} \label{eq:Poly}
L u = ( - \Delta )^m u + Q \cdot D u + qu 
\end{equation}
where we use $D = -i \nabla$ and we  
assume that the coefficients $Q$ and $q$ lie in 
Sobolev spaces of negative order. 
To give an  example of the problems we 
study, we consider the Navier boundary value problem for a polyharmonic operator 
with lower order terms
\begin{equation}\label{eq:navier}
\begin{cases} L u = 0, \qquad & \mbox{in } \Omega
\\
( -\Delta )^j u = \phi _j , \qquad & \mbox{on }\partial \Omega, j=0, \dots, m-1. 
\end{cases}
\end{equation}
For this solution $u$, we find $\frac { \partial } { \partial \nu } ( - \Delta )^j u$,
 $j=0,\dots,m-1$ where $ \partial /\partial \nu$ is the derivative in the direction of 
 outer unit normal.   The problem we study is whether the  
 coefficients $q$ and $Q$ are 
 uniquely determined by the Cauchy data  at the boundary
 $$(u, \frac { \partial u }{ \partial \nu } , 
-\Delta u, \dots, ( - \Delta )^ { m-1 } u, 
\frac \partial { \partial \nu } ( - \Delta )^ { m-1 } u  )  $$ 
for all solutions. 
Our main result gives regularity conditions on the coefficients that 
guarantee a positive answer to this question. The primary interest of our work
is that we lower the smoothness required to establish this result by up to  
one derivative as 
compared to earlier work of Assylbekov and Iyer \cite{MR4027047}.

We briefly review some of the history of inverse boundary value problems. An 
early 
study of inverse boundary value problems for second order operators was 
carried out by 
Calder\'on \cite{AC:1980} who considered a linearized problem and introduced the 
idea of looking at what we now call complex geometrical optics (CGO) solutions.  
Sylvester and Uhlmann \cite{SU:1987} obtained the first uniqueness result for the 
full  nonlinear problem.  They studied the inverse conductivity problem where the 
goal is to show that  a scalar function $ \gamma$ (the conductivity) is uniquely 
determined from 
the Cauchy data for  solutions of the operator $ \div \gamma \nabla$.
The study of  inverse boundary value problems  for the  biharmonic   operator 
was begun by Krupchyk, Lassas and Uhlmann  \cite{MR3118392} who establish 
that first order 
terms are uniquely determined by the Cauchy data. 
The stability of the solution of this problem was considered by Choudhury and 
Krishnan \cite{MR3357587}.
Krupchyk, Lassas and Uhlmann \cite{MR2873860}  study inverse boundary value 
problems for the biharmonic operators with 
partial data on the boundary. 
 Khrishnan and Ghosh \cite{MR3546596} extend this work to consider inverse  boundary value problems for 
 the polyharmonic operators with partial data.  Bhattacharyya and Ghosh
\cite{MR3953481,BG:2021} consider second order perturbations of biharmonic and polyharmonic. The most
recent paper \cite{BG:2021} establishes the uniqueness  
of a second order term from boundary data. 
Work of Yan \cite{LY:2020} establishes uniqueness for first order 
perturbations of a biharmonic operator  on transversally isotropic 
manifolds. She extends this work to give a constructive recovery 
procedure for continuous potentials \cite{LY:2021}. 
Serov \cite{MR3476498} establishes a Borg-Levinson theorem which shows that certain coefficients for biharmonic operator are determined by 
boundary spectral information. His proof proceeds by establishing a connection between the 
boundary spectral information and a Dirichlet to Neumann map. 

Our interest here is to study the regularity hypotheses that are needed to show 
that the 
coefficients are uniquely determined by the Cauchy data at the boundary. 
We review some of the history related to smoothness hypotheses needed to obtain uniqueness
in inverse boundary value problems for second 
order equations. 
An important technique 
in the study of 
inverse boundary value problems for the conductivity equation $ \div \gamma \nabla$ 
is a change
of dependent variable which transforms the conductivity equation into a 
Schr\"odinger equation. 
A calculation shows that if $ \div \gamma \nabla u =0$, then 
$ v = \sqrt \gamma \, u$ will be a solution of  $( -\Delta + q ) v =0$ 
with $q= \Delta \sqrt \gamma / \sqrt \gamma $.  Thus results for the conductivity 
equation with $\gamma $ having $s$ derivatives is implied by a result for 
Schr\"odinger operators 
with potentials that have $s-2$ derivatives, at least for $s\geq 1$.   
It will be helpful to keep this correspondence
in mind in the discussion below.  For the conductivity 
equation,  work of 
Brown  \cite{RB:1994a} and Brown and Torres \cite{BT:2003} establish the 
uniqueness of the coefficient $ \gamma$ in the conductivity  operator 
 when $ \gamma $ has roughly 3/2 derivatives.
Paivarinta, Panchenko and Uhlmann \cite{MR1993415}  construct CGO solutions 
for $C^1$-conductivities but are not able to use these solutions to 
establish a uniqueness result. 
An important breakthrough  was made by B. Haberman and D.~Tataru
\cite{MR3024091} 
who establish uniqueness for $C^1$-conductivities. 
Haberman and Tataru's innovation is to develop a family of spaces and show that in an
average sense the behavior of the potentials is better than one might expect from looking at the 
pointwise behavior. This averaging 
gives better estimates than obtained, for example, 
by P\"aiv\"arinta, Panchenko and Uhlmann and leads to  uniqueness for
$C^1$-conductivities. Further extensions  include work of 
Caro and Rogers \cite{MR3456182}  who establish uniqueness
for Lipschitz conductivities and an additional work of Haberman \cite{MR3397029}  that
considers
conductivities with one derivative in $L^d$ (where $d$ is the dimension) at least 
for $d=3,4$. Recent extensions of these ideas may be found in articles
by Ham, Kwon and Lee \cite{MR4273826} and Ponce-Vanegas
\cite{poncevanegas2019bilinear,MR4112132}. 
It is a conjecture of Uhlmann that  the least regular Sobolev space for
which we can prove uniqueness is $W^{1,d}$.  However, there are  few tools for treating conductivities
which have less than one derivative 
when $ d \geq 3$ nor is there an obvious path to constructing examples. 
If $d=2$,  Astala and  P\"aiv\"arinta
\cite{AP:2006} establish uniqueness for conductivities that are only bounded and 
measurable. 

The key innovation of our research
is
to adapt the arguments of Haberman and Tataru to polyharmonic operators. 
We will consider operators whose principal part is the 
polyharmonic operator  and 
include lower order terms involving at most one derivative as in \eqref{eq:Poly}.
For these operators we are able to establish uniqueness for potentials 
$q$ which lie in Sobolev spaces with smoothness index $ -s >-m/2 -1$. We believe that 
the cutoff in the smoothness index should be $-m$ for operators where the 
principal part is $( - \Delta )^m$. This is consistent with the result of 
Haberman and Tataru when $m=1$.
Our result allows us to get arbitrarily close to this cut off 
when $m =2$. However, when $m \geq 3$ we believe that there is a large gap between our results
and the optimal result.   

 The  investigation of uniqueness in inverse problems for  
polyharmonic 
operators with coefficients of negative order was begun by Assylbekov \cite{MR3627033,MR3691848} and Assylbekov 
and Iyer \cite{MR4027047}. In the work of Assylbekov and Iyer,   
the authors establish uniqueness 
when $ q $ lies in the
Sobolev space $W^{-s,p}$ with $ s < m /2$ and $p \geq 2d/m$ and $Q $ 
lies in a Sobolev space $W^{-s+1. p}$ with $ s < m/2$ and $p  > 2d/m $, 
at least in the case $m<d$. We refer the interested reader to the 
work of Assylbekov and Iyer for the complete statement. The  result of Assylbekov and Iyer is in the 
spirit of the work of the first author and Torres \cite{BT:2003} who studied the inverse 
conductivity problem for conductivities with 3/2 derivatives.  
A related result is due to Krupchyk and Uhlmann who consider polyharmonic
operators with potentials in $L^p$ spaces \cite{MR3484382}.  They are able to establish 
uniqueness when $q \in L^p$ for $ p > d/2m$.  Since $L^p$ embeds into $W^{ -s, q}$ for 
$ 1/p = 1/q + s/d$ our result below in Theorem \ref{th:main} will give uniqueness for 
potentials in $L^p$ for $ p > d/m$ which is a weaker result for $L^p$ spaces than
Krupchyk and Uhlmann give. This is one of several indications that our result, while an 
advance over earlier work, is not the last word on this subject. 
It also remains an open problem to extend the techniques we discuss for nonsmooth first 
order perturbations to second order perturbations as studied in the  works of 
Bhattarcharyya and Ghosh mentioned earlier.  
We refer  readers to the articles cited for
additional work on these problems. 

The main result of this work is  the following result. 
The reader will need to refer to later sections of the paper for
definitions of some of the notations occurring in this result. 

\begin{theorem} 
\label{th:main}
Fix $m \geq 2$ and $ d \geq 3$ and let $
\Omega \subset \reals ^d$ be a bounded domain 
with smooth boundary. Let $L_j = ( -\Delta)^m + Q^j\cdot D + q^j $ with $ q^j 
\in \tilde W^ { -s,p} ( \Omega)$ and each component  $Q^j$ lies in 
$\tilde W^{ -s+1,p}(\Omega)$ with $ s < m/2 +1$, $ 1/p +(s-m)/d <0 $,  and $p\geq 2$. 
If  each of the operators    $L_1$ and $L_2$ have the same linear space of Cauchy data
for 
\eqref{eq:navier}, then  
$Q^1=Q^2$ and $q^1=q^2$.
\end{theorem}

The remainder of the paper is devoted to the proof of this theorem as well as 
a  more general statement where the hypothesis involving the Cauchy data is
replaced by an assumption stated using  quadratic forms. The paper is organized as
follows. 
Section \ref{sec:notation} gives notation and definitions. Section \ref{CGO} 
describes the spaces $X^\lambda$ which were introduced by Haberman
and Tataru \cite{MR3024091} and are modeled on spaces used by Bourgain
\cite{MR1209299,MR1215780} for the study of dispersive wave equations. We consider
the operator obtained by conjugating the operator in \eqref{eq:Poly} by 
exponentially growing harmonic functions and study  its action on the 
$X^\lambda$ spaces. 
Section \ref{sec:avg}  gives an extension of the averaging method of Haberman and Tataru 
which we  use to average in $X^\lambda $ spaces. 
In section 5, we describe the Cauchy data we use to establish uniqueness and 
relate the Cauchy data to a Dirichlet to Neumann map and to  bilinear forms 
that we use in our uniqueness theorem.  In section 6, we construct 
CGO solutions. We begin with 
an analog to
the P\"aiv\"arinta, Panchenko and Uhlmann result for second order equations \cite[Theorem 1.1]{MR1993415}. We then show how the methods of Haberman and
Tataru give solutions which are better behaved.  We use these
solutions to 
give our uniqueness result. The hypothesis of this uniqueness result is
stated in terms of 
bilinear forms. We remark that the smoothness of
the domain  $\Omega$ enters only in our definition of the Cauchy data and the 
Dirichlet to 
Neumann map. Our work with 
bilinear forms only requires the domain to be Lipschitz. Finally, an appendix 
collects a few facts about Sobolev spaces that are needed to work with our
operators which have coefficients from Sobolev spaces of negative order. 

The authors would like to thank K.~Krupchyk for suggesting this problem and the 
referees 
whose
careful reading of our manuscript led to several suggestions for improvement. 

\section{Notation}\label{sec:notation}


We will let $\Omega$ be a bounded, Lipschitz domain by which we mean a bounded
connected open set whose boundary is locally the graph of a Lipschitz 
function (see, for example, Verchota \cite{GV:1984} for a definition). In a 
few places, especially section \ref{forms}, we will require the boundary 
to be smooth.  
%
We define the Fourier Transform on functions by 
\[ \hat{f}(\xi) = \int_{ \R^d} e^{-ix \cdot \xi} f(x) \, dx. \]
We will let ${\cal S}( \R ^d)$ be the Schwartz space on $\R^d$ and 
$ { \cal S}' ( \R^d)$, the dual space of tempered distributions. 
We use $\langle \cdot, \cdot \rangle$ for the bilinear pairing of 
$\cal S'$ and $\cal S$ and the pairing of duality for various Sobolev spaces. 
We will use the notation $\langle x \rangle =(1+|x|^2)^{ 1/2} $ and thus, we can write
$((1-\Delta ) ^ { s/2}u) \,\hat{} \, (\xi) = \hat u (\xi ) \langle \xi \rangle ^ { s}$. 

Our hypotheses on the smoothness of the coefficients will be given in terms 
of Sobolev spaces. 
For $s \in \R$ and $1< p < \infty$ we let $W^{s,p}(\R^d)$ 
be the Sobolev space of order $s$ with the norm
$\norm{u}_{W^{s,p}(\R^d)} = \norm{(1- \Delta)^{s/2}
u}_{L^p(\R^d)}$.
We also let $ W ^ { s,p} ( \Omega)$ denote the  Sobolev space on $\Omega$
obtained by 
 by restricting distributions  in $W^{s,p}(\R^d)$ to $\Omega$. 
 We let $W^{s,p}_0( \Omega)$ denote the closure of $C_0^ \infty ( \Omega)$ in the norm
 of $W^ {s,p} ( \Omega)$.
We define
\[ \tilde{W}^{s,p}(\Omega) = \{ u \in
W^{s,p}(\R^d) \, : \, \text{supp}(u) \subset
\bar{\Omega} \}. \]
Note that this notation is misleading since $\tilde W ^ { s,p}( \Omega)$ may not be 
a space of distributions in $ \Omega$.  It is useful to note that for Lipschitz domains 
(and some more general domains) that $ C^\infty_0 ( \Omega) $ is dense in  the space
$ \tilde W ^ { s,p}( \Omega)$. See McLean \cite[Theorem 3.29]{MR1742312} for the case $p=2$ and a proof that 
may be generalized to $p \in (1,\infty)$.  
The distinction between $W^{s,p}_0( \Omega) $ and 
$\tilde W^{s,p}( \Omega)$ is subtle. For many values
of $s$ and $p$, the spaces agree. However, in certain cases the 
space $W_0^ { s, p }( \Omega)$ will be a larger space. The norm 
in $ \tilde W ^{s,p}( \Omega)$ is the norm of a distribution which is zero 
outside 
$ \bar \Omega$.   The norm of an element $u$
in $ W^{ s,p}_0( \Omega)$ is the infinum of the norm in $\reals^d$ 
over all possible extensions $u$  to $ \reals ^d$ and may be smaller. 
The last function space we will 
need is the H\"older space $C^ \theta ( \R^d)$. For 
$ \theta \in (0,1]$, this space is the collection of functions on 
$\R^d$ for which the norm
$$
\| f\|_{ C^\theta ( \R^d)}= \sup _ { x \in \R ^d} |f(x) | + \sup _ { x\neq y } \frac { |f(x) - f(y) | } { |x-y |^ \theta}
$$
is finite. 

If the domain $ \Omega $ is Lipschitz, then we have a 
characterization of the dual space 
of $W^ { s,p}( \Omega) $,
\begin{equation} \label{eq:SobDual}
W^ { s,p} ( \Omega) ^ * = \tilde W ^ { -s, p'}( \Omega),\qquad  1<p < \infty , s\geq 0.
\end{equation}
The case $p=2$ may be found in the monograph of McLean 
\cite[Theorem 3.30]{MR1742312} and his argument extends to $p \in (1,\infty)$. 
The spaces $ \tilde W ^ { s,p}( \Omega)$ will be  useful for the study of 
solutions with nonzero boundary data since for certain $ s$ and $p$ we  will be
able to define the product of a function in a Sobolev space of positive order with 
the coefficients taken from Sobolev spaces of negative order. 
From the Sobolev embedding theorem and the product rule, we have that  the bilinear map
$(u,v) \rightarrow v D^ \alpha u $ maps $ W^ { m,2}( \Omega ) \times 
W^ { m,2} ( \Omega) \rightarrow W^ { m-|\alpha |, t} ( \Omega) $ 
provided $|\alpha | \leq m$ and 
where $ t$ is given by 
\begin{equation}\label{eq:rDef}
 \begin{cases} \frac 1   t = 1 - \frac m d , \qquad & \frac m d <  \frac 1 2 \\
\frac 1 t = \frac 1 2  , \qquad & \frac m d > \frac 1 2 \\
\frac 1 t > \frac  1 2 , \qquad & \frac m d = \frac  1 2 . 
\end{cases}
\end{equation} 
This is a special case of a more general result in the monograph of Runst and Sickel \cite[\S4.4.4, Theorems 1,2]{MR98a:47071}.
Thus, if $f \in \widetilde W ^ { |\alpha | - m , t'} ( \Omega) $ and $u \in W^ { m,2} ( \Omega)$, then we 
may define $f D^ \alpha u  \in \tilde W ^ { -m,2} ( \Omega)$. 
If $v \in W^ { m, 2}( \Omega)$, then we define $\langle f D^ \alpha u, v \rangle  $ as equal to 
$ \langle f ,v D^ { \alpha } u \rangle$ and we  have the estimate  
\begin{equation} \label{eq:trilinear} 
|\langle f D^ \alpha u , v \rangle | \lesssim 
\|f\| _{\widetilde W ^ { |\alpha | -m, t'}( \Omega)} 
\| u \| _ {W^ { m,2}( \Omega) } \| v \|_ {W^ { m,2} ( \Omega)}. 
\end{equation} 
We now describe a standard notion of weak solution that we will use to give a
careful definition of solutions to the equation $ Lu =0$. 
We let 
$B_0 : W^ { m,2} ( \Omega) \times W^{m,2} ( \Omega) \rightarrow \complexes  $ be a form 
given by 
$$
B_0( u, v) = \int _ { \Omega} \sum_ {|\alpha |=| \beta|=m} 
a_{ \alpha, \beta } D^ \alpha u D^ \beta v\, dx 
$$
with the property that
\begin{equation}
    \label{eq:FormProp}
 B_0(u,v) = \int_\Omega [(- \Delta)^m u ] v \, dx, 
\qquad u \in C^ \infty ( \Omega), v \in C_0 ^ \infty ( \Omega).   
\end{equation}
There are two common choices for the form $B_0$. The first is useful for 
the Navier boundary value problem 
\begin{equation} \label{eq:FormNavier}
B_0(u,v) = \begin{cases}
\displaystyle \int_\Omega (- \Delta)^{m/2} u (- \Delta)^{m/2} v \, dx & \text{if $m$ is even} \\[0.2in]
\displaystyle \int_\Omega \nabla (- \Delta)^{(m-1)/2} u \cdot 
\nabla (- \Delta)^{(m-1)/2} v \, dx 
    \quad    & \text{if $m$ is odd}
\end{cases} \end{equation}
For some purposes, the  form
\begin{equation}
    \label{eq:FormCoercive}
    B_0(u,v) = \int _ \Omega \sum_ { |\alpha | =m } \frac { m!}{ \alpha !} 
    D^ \alpha u D^ \alpha v \, dx. 
\end{equation}
is more convenient because it is easier to see that this form is coercive. 
We note that in our uniqueness theorem based on forms, Theorem \ref{th:Unique},  will hold for any choice of $ B_0$ satisfying \eqref{eq:FormProp}. 

To define weak solutions to the equation $Lu = 0$ we require that 
\begin{equation}\label{eq:Qprop}
Q \in \tilde W ^ { 1-m, t'} ( \Omega) \quad \mbox{and} 
\quad q \in\tilde  W^ { -m,t'}( \Omega) , \quad \mbox{with $t$ as in \eqref{eq:rDef}.}
\end{equation}
We say $u$ is a {\em weak solution }of
\[  Lu =  (- \Delta)^m u + Q \cdot D u + qu = 0 \]
if $u \in W^{m,2} ( \Omega) $ and 
\[ B_0(u,v) + \langle Q \cdot D u , v \rangle + \langle qu , v \rangle = 0, 
\qquad  v \in W^{m,2}_0(\Omega) . 
\]
The estimate \eqref{eq:trilinear} and the assumption \eqref{eq:Qprop} 
guarantee that the bilinear form 
for  the equation is continuous on
$ W^ { m,2} ( \Omega) \times W^ { m,2} ( \Omega) $.  We note that the condition 
\eqref{eq:Qprop} is close to the weakest condition that allows us to 
discuss weak solutions of the equation 
$Lu =0$; we will need to impose additional regularity on $Q$ and $q$ 
in order to study the inverse problem that is the  focus of this paper.

We will adopt the convention of  using $ a \lesssim b$ to mean
that there is a constant $C$ so that $ a \leq C b$. The constants $C$ 
are allowed to depend on the $L^p$-index,  the order of a Sobolev space, 
the dimension, the domain and the coefficients $Q$ and $q$ appearing in 
the operator. However, it is an important point that $C$ will 
always be independent of $h$ for $h $ small as this is needed to evaluate limits as $h$ tends to zero. We use $ a \approx b$ to mean $ a \lesssim b $ and $ b \lesssim a$. 

\section{CGO solutions and $X^ \lambda $ spaces }\label{CGO}

In this section, we begin the construction of 
 CGO solutions of \eqref{eq:Poly}. 
We let  $\cal V$ denote the collection of vectors
\begin{equation} 
\label{eq:zetadef}
{\cal V} 
= \{  \zeta \in \C ^d :  \zeta =  \mu_1 + i \mu_2, \  \mu_j \in \R ^d, 
\ \mu_i \cdot \mu _j = \delta _{ ij}\}. 
\end{equation}
For $ \zeta \in \cal V$ and $ h\neq 0$, the function  $ e^ { x \cdot \zeta /h}$ will be harmonic. 
For $0 < h \leq h_0 < 1$ we will construct solutions  of $ Lu =0 $
of the form
\begin{equation} \label{eq:CGO}
u(x) = e^{x \cdot \zeta / h} (a(x) + \psi(x)) \text{ in } \Omega
\end{equation}
where the amplitude $a$ will be smooth. 
If we substitute \eqref{eq:CGO} into the equation $Lu = 0$ and
multiply on the left by $h^{2m} e^ { - x\cdot \zeta / h}$, 
we obtain an equation for $ \psi $
\begin{equation} 
\label{eq:PDE2}
L_ \zeta \psi = - L_ \zeta a.
\end{equation} 
We introduce the operator $P_ \zeta (hD)$ defined by
\[  P_\zeta (hD) = P(hD)= e^{- x \cdot \zeta / h} (- h^2 \Delta ) e^{x \cdot \zeta / h}, \]
where we will view this as a semi-classical operator with symbol 
$P_\zeta ( \xi ) = P ( \xi) = |\xi|^2 - 2 i \zeta \cdot \xi $. 
With this notation we may write 
\begin{equation} \label{eq:condef}
L_\zeta 
= P(hD)^m + h^{2m} Q \cdot (\frac{\zeta}{ih} + D) + h^{2m} q. 
\end{equation}
We will give estimates for solutions of the equation 
$ L_ \zeta u =f$, at least for $h$ small. 
Since the symbol of the operator $P(hD)$ vanishes on a set of codimension 2, the operator 
is not elliptic and finding solutions is an interesting problem. 
%
To study the equation \eqref{eq:PDE2}, it is useful to use function spaces that work well with $P(hD)$. 
Similar spaces were  used by Bourgain \cite{MR1209299} in his
study of evolution equations.  The spaces described below were introduced in the 
study of the inverse conductivity problem by Haberman and Tataru \cite{MR3024091}. 
For $ \lambda \in \R$, $ \zeta \in {\cal V}$, and $h>0$, we introduce a space $X_{h \zeta}^ \lambda $ 
by 
\[ X^ \lambda_{h \zeta}= X^\lambda = \{ u \in \mathcal{S}'(\R^d) : \hat{u} 
\text{ is a function and } \norm{u}_{X^\lambda} < \infty \} \]
where the norm in $X^\lambda$ is given  by
\[ \norm{u}_{X^\lambda}^2 
= \int_{ \R^d} |\hat{u}(\xi)|^2 (h + |P_\zeta (h \xi)|)^{2 \lambda} \, d \xi .  \]
We note that we will frequently omit the dependence of the operator $P_ \zeta(hD)$ 
on $ \zeta$ and the dependence of the space $X^ \lambda$ on $h $ and $\zeta$ when these
can be
determined from the context. 
The dual space of $X^\lambda$ 
is $(X^\lambda)^* = X^{- \lambda}$.
Since the weight $(h + |P(h \xi) | ) ^ {\lambda} $ is bounded above and below on
compact sets, 
it is clear that the
Schwartz space is dense   in $X^ \lambda$. 
%

Our next goal is to construct an operator that is a right inverse to $P(hD)$ on 
$ \Omega$. 
While we are interested  in constructing solutions in $ \Omega$, our
right inverse will
act on the spaces $X^\lambda$ which consist
of distributions on $ \R ^d$. To construct our operator, we 
choose  $\vp \in C_0^\infty(\R^d)$ with 
$\vp = 1$ in a neighborhood of $\bar{\Omega}$. 
Let $J_\vp = \vp |P(hD)|^{-1/2}$ and $I_\vp = \vp P(hD)^{-1} \vp$. 
We then have $I_\vp = J_\vp \frac{|P(hD)|}{P(hD)} J_\vp^*$. What is left to 
show is that $I_\vp$ is a bounded operator between suitable spaces and a right 
inverse.

We begin by recalling  an estimate from Haberman and Tataru 
 which they use to give an estimate for $u$ in 
$L^2 _ { loc}$ in terms of  \( |P(hD)|^ { 1/2} u\). 
\begin{lemma}[{\cite[Lemma 2.1]{MR3024091}}] \label{lm:Weighted Bound}
Let $v$ and $w$ be nonnegative weights on $\R^d$. If $\vp \in {\cal S} ( \R^d) $, 
then
\[ \norm{\vp * f}_{L^2_w} \lesssim_\phi \min \Big( \sup_\xi \int J(\xi, \eta) \, 
d \eta , \, \sup_\eta  \int J(\xi, \eta) \, d \xi \, \Big)^{1/2} \norm{f}_{L^2_v}, \]
where
\[ J(\xi, \eta) = |\vp(\xi - \eta)| \frac{w(\xi)}{v(\eta)} . \]
\end{lemma}


\begin{lemma}
\label{lm:Brown Notes}
Let $ h>0$ and $ 0 \leq s \leq 1$. We have
\begin{equation} 
\int _ { B(x,r) } \frac { |\xi |^s }{|P(h\xi)|} \, d \xi \lesssim \frac { r ^ { d+1}}{ h ^ { 1+s}}, \qquad x \in \reals ^d ,\  r >0 . \label{RB1}
\end{equation}
If $ \phi \in { \cal S}( \reals ^d)$, then 
\begin{equation} 
\int _ { B(x,r) } \frac {|\phi ( x-\xi)|  |\xi |^s }{|P(h\xi)|} \, d \xi \lesssim \frac { 1} { h ^ { 1+s}}. \label{RB2} 
\end{equation}
\end{lemma}

\begin{proof}
We begin with  the proof of \eqref{RB1}. By applying the change of variables, 
$ \xi \rightarrow \xi /h$ we may reduce to the case where $ h =1$. To show \eqref{RB1} 
with $h =1$, we will consider three cases:
a)   $ |x | < 8$, $ r\leq 1$, b) $|x |\geq 8$,  $ r \leq 1$, c) $ r >1$.   
We let $ \Sigma = \{ \xi : P(\xi ) = 0 \}$ 
and since $P(\xi) = |\xi + \im \zeta \cdot \xi |^2 -1 -2i \re \zeta \cdot \xi $, we have that  
$\Sigma$ is a sphere of radius 1 and codimension two centered 
at $- \im \zeta$. Furthermore, 
one can check that $ \nabla P(\xi) \neq 0 $ on $ \Sigma$ and thus it follows that 
$ P( \xi ) \approx \dist( \xi , \Sigma) $ on $ B(0, 9)$. To establish \eqref{RB1} in case a) we need 
to estimate integrals of $ |\xi |^s / | P (\xi)| \sim 1/\dist( \xi , \Sigma) $. We can obtain
the necessary estimate by using a change of variables to straighten out a portion of the sphere 
to reduce to integrating the distance to a plane of codimension two  in $ \R ^d$. 

To establish \eqref{RB1} in case b), we use that $|\xi|^s/ |P(\xi) | \approx |\xi |^ {s-2} \lesssim 1$ 
for $ |\xi | \geq 7$. From this, we have
$$ \int _ { B(x,r) } \frac { |\xi |^s} { |P(\xi) |} \, d\xi \lesssim r^d \leq  r^ { d-1}
$$
where the last inequality uses that $ r\leq 1$. Finally, in case c) we use that $ |P(\xi) |\approx |\xi|^2$ for $|\xi |$ large to obtain 
$$
\int _ { B(x,r) } \frac { |\xi |^s}{ |P( \xi) |} \, d \xi 
\leq \int _ { \{|\xi | \leq 8\}} \frac { |\xi |^s}{ |P( \xi) |} \, d \xi 
+ \int _ { \{|\xi| > 8\} \cap B( x, r)  } |\xi|^ { s-2} \,d \xi .
$$
We use  case a) to bound the first term and obtain that the previous display is at most
 $ 1 + r ^{ d-2+s} \lesssim r^ { d-1}$. The last inequality follows since $ r \geq 1 $ 
 and $ s \leq 1$. 
We note that if $ d=2$ and $s=0$, the intermediate estimate involves a logarithm, but the conclusion still holds. 

We turn to the proof of 
\eqref{RB2} and give the proof for  $h>0$. 
We write  $ \R ^d = \cup _ { k \geq 0} A_k$ where $ A_ 0 = B(x,1)$ 
and  for $k \geq 1$
$A_k = B(x, 2^k ) \setminus B(x, 2^ { k-1})$.  
If $ \phi \in { \cal S }( \R^d)$, then 
it is rapidly decreasing and, in particular, we have 
$|\phi(x- \xi) | \lesssim 2^ { -kd} $ if $ \xi \in A_k$. 
We break the integral in \eqref{RB2} into integrals over the sets $A_k$, and use \eqref{RB1} to obtain
$$
\int _ { \R^d } \frac {| \phi (x- \xi) | |\xi |^s}{ |P( \xi) |} \, d \xi 
\lesssim  \frac 1 { h ^{ 1+s}}\sum _ { k =1 } ^ \infty 2^ { k( d-1) } \cdot 2^ { -kd} 
\lesssim \frac 1 { h ^ { 1+s}}.
$$
\end{proof}

\begin{lemma} 
\label{lm:Case2} 
Let $ \phi \in {\cal S } ( \R ^d)$. 
For any $\eta \in \R^d$,
\[  \int_{\R^d} \frac{| \varphi(\xi - \eta)| }
{ |P(h \xi)| } ( \big|| P(h \eta)| - |P(h \xi)|  \big| ) \, d \xi
 \lesssim 1  \]
\end{lemma}

\begin{proof}
Using the triangle inequality and the definition of  $P(h \xi)$, we have
\begin{equation} \label{Case2.1}
\big|| P(h \eta)| - |P(h \xi)| \big | \lesssim h |\xi - \eta|(h |\xi + \eta| + 1).
\end{equation}
Using \eqref{Case2.1}, we obtain the following inequality.
\begin{align*}
 \int_{\R^d} \frac{ |\varphi(\xi - \eta) | } { |P(h \xi)| }
( \big|| P(h \eta)| - |P(h \xi) | \big | ) \, d \xi   \lesssim &  \int_{\R^d} 
\frac{ |\varphi(\xi - \eta)| } { |P(h \xi)| }h |\xi - \eta| \, d \xi \\
&\quad + \int_{\R^d} \frac{ |\varphi(\xi - \eta)| } { |P(h \xi)| } h^2 |\xi - \eta|^2 \, d \xi \\
&\quad + \int_{\R^d} \frac{ |\varphi(\xi - \eta)| } { |P(h \xi)| } h^2 |\xi - \eta||\xi| \, d \xi.
\end{align*}
Using Lemma \ref{lm:Brown Notes}  and recalling that $ 0 < h \leq 1$ 
completes the proof.
\end{proof}

We now have enough tools to show that $I_\vp$ is a bounded 
right inverse. We begin with an estimate for $J_\vp$ which will be 
the main step in establishing estimates for $I_ \vp$.

\begin{theorem} \label{th:Jphi}
The operator
$J_\vp : X^{\lambda} \to X^{\lambda + 1/2}$ is a bounded operator for all $\lambda \in
\R$. 
\end{theorem}

\begin{proof}
We have that 
$$
(J_\phi u) \hat {} = \hat \vp * ( \hat u / |P( h \xi) |^ { 1/2}). 
$$
Since $ \hat \vp$ is again a Schwartz function, we replace  $ \hat \vp$ by $ \vp$, and we will show that for a Schwartz function, $\vp$, we have 
\begin{equation} \label{eq:jpStep1} \sup_\eta \int_{\R^d} \frac{| \vp (\xi - \eta)|}{|P(h \xi)|} 
\frac{(h + |P(h \xi)|)^{2\lambda + 1}}{(h + |P(h \eta)|)^{2 \lambda}} \, d \xi 
\lesssim 1, \qquad  0 \leq h \leq 1.  
\end{equation}
Then an application of Lemma \ref{lm:Weighted Bound} will give the 
result of the Theorem. 

To establish \eqref{eq:jpStep1}, we begin by adding and subtracting $(h + |P(h \eta)|)^{2\lambda + 1}$ 
and then adding and subtracting $|P(h \xi)|$, we can show the following inequality. For $\eta \in \R^d$,
\[\begin{split}
\int_{\R^d} &\frac{| \varphi(\xi - \eta)| }{|P(h \xi)|} \frac{(h + |P(h \xi)|)^{2\lambda + 1}}
{(h + |P(h \eta)|)^{2 \lambda}} \, d \xi \\
&\leq \int_{\R^d} 
    \frac{ |\varphi(\xi - \eta)| }{|P(h \xi)|}
\, h \, d \xi 
+ \int_{\R^d} \frac{ |\varphi(\xi - \eta) | }{|P(h \xi)|} \, ( ||P(h \eta)| - |P(h \xi)||) \, d \xi 
+ \int_{\R^d} | \varphi(\xi - \eta)| \, d \xi \\
&\quad + \int_{\R^d} \frac{ |\varphi(\xi - \eta) |}{|P(h \xi)|} \, 
\frac{| (h + |P(h \xi)|)^{2 \lambda + 1} - (h + |P(h \eta)|)^{2 \lambda + 1} | }
{ (h + |P(h \eta)|)^{2 \lambda}} \, d \xi \\
&= I + II + III + IV.
\end{split}
\]
Using  \eqref{RB2}  
of Lemma \ref{lm:Brown Notes} we see that $I$ is bounded by a 
constant which is independent of $\eta$ and $h$, Lemma \ref{lm:Case2}  gives the same 
estimate for $II$, and the estimate for $III$ is easy. This 
leaves us to bound $IV$. 

Now, we turn to an application of the mean value
theorem. We set $f(t) =  t^{2 \lambda + 1}$ for $ t \geq 0$. For $\eta, \xi \in \R^d$, let 
$a = \min ( \, h + |P(h \xi)|, \, h + |P(h \eta)| \, )$ and 
$b = || P(h \xi)| - |P(h \eta)| |$. Then, we get that 
$a + b  = \max ( \, h + |P(h \xi)|, \, h + |P(h \eta)| \, )$. The mean 
value theorem states that there exists $\theta$ between $a$ and $a+ b$ such that
\[ f(a+b) - f(a) = (2 \lambda + 1) \theta^{2 \lambda} b. \]
We will go through the proof when $\lambda \geq 0$ as the proof for 
$\lambda < 0$ is similar. Let $A_0 = B(\eta,1)$ and 
$A_k = B(\eta,2^k) \setminus B(\eta, 2^{k-1})$ for $k \geq 1$. 
Using \eqref{Case2.1} and then the mean value theorem, we obtain the upper bound
\begin{align*}
    IV &= \int_{\R^d} \frac{|\varphi(\xi - \eta)|}{|P(h \xi)|} \, 
    \frac{|(h + |P(h \xi)|)^{2 \lambda + 1} - (h + |P(h \eta)|)^{2 \lambda + 1} |}{(h + |P(h \eta)|)^{2 \lambda}} \, d \xi \\
    & \leq \sum_{k=0}^\infty \int_{A_k} \frac{|\varphi(\xi - \eta)| }{|P(h \xi)|} \frac{\theta^{2 \lambda} h|\xi - \eta|(h|\xi + \eta| + 1) }{(h + |P(h \eta)|)^{2 \lambda}} d \xi . 
\end{align*}
We use 
$|\theta| \leq h + |P(h \eta)| + |P(h \xi)|$ and that 
$\lambda \geq 0 $ to obtain an upper bound for $\theta^{2 \lambda }$. We now consider  
four cases.  Since $\vp $ is a Schwartz function 
 for  $N = 2 +d + 4 \lambda$, we have
$|\vp(\xi)| \lesssim \langle \xi \rangle ^ { -N}$. 

\textbf{Case 1:} $|h \eta| < 8$ and $2^k \leq 32/h$. Then, in $A_k$, we have $|\xi + \eta| \leq 48/h$, $|P(h \xi)| \lesssim |P(h \eta)| + h2^k$, and $|\theta|^{2 \lambda} \lesssim (h2^k + |P(h \eta)|)^{2\lambda}$. This gives us
\begin{align*}
    \int_{A_k} \frac{|\varphi(\xi - \eta)|}{|P(h \xi)|} \frac{\theta^{2 \lambda} h|\xi - \eta|(h|\xi + \eta| + 1) }{(h + |P(h \eta)|)^{2 \lambda}} d \xi & \lesssim \int_{A_k} \frac{|\varphi(\xi - \eta)|}{|P(h \xi)|} \frac{(h2^k + |P(h \eta)|)^{2 \lambda}h2^k}{(h + |P(h \eta)|)^{2 \lambda}} d \xi \\
    & \lesssim 2^k 2^{-Nk} \int_{A_k} \frac{h }{|P(h \xi)|} \Big( \frac{h2^k + |P(h \eta)|}{h + |P(h \eta)|} \Big)^{2 \lambda} d \xi \\
    & \lesssim 2^k 2^{-Nk} \int_{A_k} \frac{h }{|P(h \xi)|} \Big( 2^k \Big)^{2 \lambda} d \xi \\
    & \lesssim 2^{k(-N + 2\lambda + d)}.
\end{align*}

\textbf{Case 2:} $|h \eta| < 8$ and $2^k \geq 32 / h$. Then, in $A_k$, we have $2^{k-1} \geq 16/h$, $16/h \leq |\xi - \eta|$, $|\xi| \geq 8/h$, and $h|\eta + \xi| \leq h2^k + 8$. Using similar techniques to Case 1, we obtain
\begin{align*}
    \int_{A_k} \frac{|\varphi(\xi - \eta)|}{|P(h \xi)|} \frac{\theta^{2 \lambda} h|\xi - \eta|(h|\xi + \eta| + 1) }{(h + |P(h \eta)|)^{2 \lambda}} d \xi & \lesssim  2^{k(-N + d + 1 + 4 \lambda)}.
\end{align*}

\textbf{Case 3:} $|h \eta| \geq 8$ and $2^k \leq |\eta| / 8$. 
Then, $A_k \subset \{ \xi : 7|\eta|/8 \leq |\xi| \leq 9|\eta|/8 \}$.
Therefore, in $A_k$, we have $|\eta| \approx |\xi|$. Therefore we have 
$1 \leq h|\xi|/8$, $1 + h|\xi + \eta| \lesssim h |\xi|$, and 
$|P(h\eta)| \approx |h \eta|^2 \approx |h \xi|^2 \approx |P(h \xi)|$. 
We then obtain
\begin{align*}
     \int_{A_k} \frac{|\varphi(\xi - \eta)| }{|P(h \xi)|} \frac{\theta^{2 \lambda} h|\xi - \eta|(h|\xi + \eta| + 1) }{(h + |P(h \eta)|)^{2 \lambda}} d \xi \lesssim 2^{k(-N + 1 + d)}.
\end{align*}

\textbf{Case 4:} $|h \eta| \geq 8$ and $2^k \geq |\eta| / 8$. Then, in $A_k$, 
we have $|\xi| \lesssim 2^k$, $|P(h\xi)| \lesssim h2^k + h^2 2^{2k}$, 
$|P(h\eta)| \lesssim h2^k + h^2 2^{2k}$, and $|\eta + \xi| \lesssim 2^k$. 
We then obtain
\begin{align*}
     \int_{A_k} \frac{|\varphi(\xi - \eta)|}{|P(h \xi)|} 
     \frac{\theta^{2 \lambda} h|\xi - \eta|(h|\xi + \eta| + 1) }{(h + |P(h \eta)|)^{2 \lambda}} d \xi \lesssim 2^{k(-N + 1 + 4\lambda + d)}.
\end{align*}
 Our choice of $N$ gives 
\begin{align*}
    \sum_{k=0}^\infty \int_{A_k} \frac{|\varphi(\xi - \eta)|}{|P(h \xi)|} 
    \frac{\theta^{2 \lambda} h|\xi - \eta|(h|\xi + \eta| + 1) }
     {(h + |P(h \eta)|)^{2 \lambda}} d \xi 
     \lesssim \sum_{k=0}^\infty 2^{-k}.
\end{align*}
This gives us that
\[ \sup_\eta \int_{\R^d} \frac{|\varphi(\xi - \eta)|}{|P(h \xi)|} 
\frac{(h + |P(h \xi)|)^{2\lambda + 1}}{(h + |P(h \eta)|)^{2 \lambda}} \, d \xi < \infty \]
which completes our proof of \eqref{eq:jpStep1}  when $ \lambda \geq 0$. 
We leave it to the reader to 
adapt this argument to the case $ \lambda \leq 0$. 
\end{proof}

\begin{corollary} \label{cor:Iphi} The map 
$I_\vp : X^\lambda \to X^{\lambda + 1}$ is a bounded operator for all $\lambda \in \R$.
\end{corollary}

\begin{proof}
Recall that $I_\phi = J_ \phi \frac{|P(hD)|}{P(hD)} J_ \phi ^ *$. From 
Theorem \ref{th:Jphi} we have  and $ J_\phi: X^ { \lambda + 1 /2} \to X^ {\lambda+1} $. 
By passing to the adjoint, the boundedness of 
$J_ \phi : X^ { -\lambda -1 /2} \rightarrow X^ { - \lambda}$ implies that 
 $J_\phi^ * : X^ \lambda \rightarrow X^ { \lambda + 1/2}$. 
Since  $\frac{|P(hD)|}{P(hD)}$ has a symbol of modulus one, 
it is an isometry on $X^\lambda$ for all $\lambda$. It follows 
that $I_\vp : X^\lambda \to X^{\lambda + 1}$.
\end{proof}

The next Lemma will imply  that $I_\vp$ is a right inverse to $P(hD)$ on $\Omega$. 

\begin{lemma} \label{lm:Inverse} 
We have 
$P(hD)I_\vp u = u$ in $\Omega$.
\end{lemma}

\begin{proof}
First, note that $D \vp = 0$ in $\Omega$ and $\vp^2 = \vp = 1$ in $\Omega$. 
Thus,  we see that
\begin{align*}
[hD, \vp]u = hD(\vp \, u) - \vp h Du = u hD \vp + \vp h Du - \vp h Du = u hD \vp = 0,
\qquad \mbox{in } \Omega
\end{align*}
Since  $[hD,\vp] = 0$ in $\Omega$, we have
\begin{align*}
P(hD)I_\vp u = P(hD) \vp P(hD)^{-1} \vp u = \vp P(hD) P(hD)^{-1} \vp u = \vp^2 u 
= u, \qquad \mbox{in } \Omega
\end{align*}
as we wanted.
\end{proof}

We now construct solutions to the   equation 
$L_ \zeta \psi = f$ where the operator $ L_\zeta$ is introduced in
\eqref{eq:condef}. 
We consider the integral equation  
\begin{equation} 
\label{eq:intfrm}
\psi + h^{ 2m} I_\vp^m ( Q\cdot ( \frac \zeta { ih} +D ) \phi + q\psi ) = 
 I _ \vp ^m f 
\end{equation} 
and will show that the map 
\begin{equation} \label{eq:contraction}
   \psi \rightarrow 
   h^{2m}I_\vp^m ( Q \cdot ( \frac{\zeta}{ih} + D ) \psi ) 
   + h^{2m} I_\vp^m( q \psi )
\end{equation}
is a contraction on $X^ { m/2}$ for suitable conditions on $q$, $Q$, 
and $h$.  Recalling Lemma \ref{lm:Inverse}, we can show that a 
solution of our integral equation is a distribution solution of the equation 
$ L_ \zeta \psi =f $. Since \eqref{SSE2} below implies that $\psi$ is in 
$W^ { m,2} ( \R^d)$, also obtain that $\psi$ is a weak solution. 
To prove we have a contraction map, we need estimates for the product 
$ q \psi$ when $q$ is a 
distribution of negative order and $ \psi $ lies in $X^ \lambda$ for 
some $\lambda >0$. 


We begin with a simple, but useful estimate that will allow us to 
relate smoothness in the
spaces $X^ \lambda$ with smoothness in $L^2$-Sobolev spaces. 
Suppose $  0< h \leq 1$ and $ 0 \leq s \leq 2\lambda $, then we have
\begin{equation} \label{eq:supremum}
  \sup _ { \xi \in \R^d} \frac { \langle \xi \rangle ^{s} }
  {( h + | P(h\xi)|)^\lambda } \lesssim h^ { -\lambda -s}. 
\end{equation}
To establish \eqref{eq:supremum}, we consider the case
$|h\xi| \geq 8$ and $ |h\xi|<8$ separately. If $ |h \xi | \geq 8$, 
we have $ |P(h\xi)| \approx h^2 |\xi|^2$ and thus
$\langle \xi \rangle ^s/( h + |P(h\xi)|) ^ \lambda 
\lesssim \langle \xi \rangle ^ { s -2 \lambda} h^{ -2\lambda}\lesssim h^{-s}$.
Since $0<h\leq 1$, this  is better than the estimate \eqref{eq:supremum}. In the 
case that $ |h\xi | < 8 $, we that $\langle \xi \rangle ^s \lesssim h^ { -s}$
and that $ (h + |p(h\xi) | ) ^ { -\lambda } \lesssim h^{- \lambda}$. Multiplying 
these two inequalities leads to \eqref{eq:supremum}. 
Our next step uses the estimate \eqref{eq:supremum} to establish estimates 
relating 
$L^2$-Sobolev spaces and  $X^ \lambda $ spaces. 

\begin{proposition}\label{prop:Df}
  If $ |\alpha | \leq 2 ( \lambda _2 -\lambda _1)$, then
  \begin{equation}\label{SSE1}
    \| D^ \alpha u \|_{X^{ \lambda_1 }} \lesssim h^ { - |\alpha | + \lambda _1 -
      \lambda _2 } \|u \|_ { X^ { \lambda _2 }}.
  \end{equation}
  and if $ 0 \leq s \leq 2 \lambda$, then
  \begin{align}\label{SSE2}
    \|u \| _ { W^ { s, 2 } ( \reals ^ d ) } & \lesssim h ^ {- s- \lambda } \| u \|_{ X^ \lambda } \\
\label{SSE3}
    \| u \|_{ X^{- \lambda }}  
    & \lesssim h ^ {- s- \lambda }        \|u \| _ { W^ { -s, 2 } ( \reals ^ d ) }.
  \end{align}
\end{proposition}
\begin{proof}
  We begin with the proof of \eqref{SSE1}.  We have
  $$
  \| D^ \alpha u \|^2 _ { X^ { \lambda _1} } \leq \sup_{\xi \in \R ^d} 
  |\xi|^{2 |\alpha |}
  (h+ |P(h\xi)|) ^ { 2 ( \lambda _1 - \lambda _2 ) } \| u \| _ { X^ { \lambda _2 } }^2
  \lesssim  h ^ { 2( - |\alpha | + \lambda _1 - \lambda _2 )}
    \|u \|_ { X^ { \lambda _2 } } 
  $$
  where the second inequality uses \eqref{eq:supremum}.

The proof of \eqref{SSE2} is similar. Since the embedding of $W^{-s,2}( \reals ^d)$ into  
$X^ { - \lambda} $ 
 is the adjoint of the embedding 
  $ X^ \lambda \subset W^ { s,2} $, the estimate \eqref{SSE3} follows from \eqref{SSE2}. 
\end{proof}

\begin{proposition}
\label{prop:newcon}
Suppose $ f \in L^ \infty ( \reals^d)$, $ \zeta$,  $\zeta_1$ and $\zeta _2 $ are in ${{\cal V}}$ (see 
\eqref{eq:zetadef}) and $0<h\leq 1$. 
Then when $ 0 \leq |\alpha | + |\beta| \leq 2 \lambda$  we have the trilinear estimate
\begin{equation} 
\label{eq:trilinear2}
|\langle D^ \alpha f D^ \beta u, v \rangle  | \lesssim 
h^ {-2 \lambda -|\alpha | -  |\beta|  } \| f \| _ \infty 
\| u \| _{ X^ \lambda_{ h\zeta_1}}
\| v \| _{X^ \lambda _ { h \zeta _2 }} .
\end{equation}

Let $ Tu=  (D^ \alpha f)  (D+ \frac \zeta { ih} ) ^ \beta u $ 
  with $ f  \in L^ \infty$ and assume 
  $ | \alpha | + |\beta| \leq 2 \lambda $, then
  \begin{equation} \label{eq:map} 
    \| Tu \| _{X_{h \zeta_2}^ { -\lambda }} 
    \lesssim \| f\| _ \infty  h ^ { -2 \lambda
      - |\alpha | - |\beta | } \| u \|_ { X_{h \zeta _1}^ \lambda } .
  \end{equation}
If in addition, we have $f \in C^ \theta ( \reals ^d )$ and 
$ |\alpha | \geq 1$, then we have
    \begin{equation}\label{eq:holder} 
    \| Tu \| _{X^ { -\lambda } _{h \zeta_2} }\lesssim \| f\| _ {C^ \theta ( \reals^d)}  
    h ^ { -2 \lambda
      - |\alpha | - |\beta | + \theta } \| u \|_ { X^ \lambda _{h \zeta_1}}.
  \end{equation}
\end{proposition}

\begin{proof} We first consider the case where $ u $ and $v$  
are in the Schwartz class. 
Using the definition of  $D^ \alpha f D^ \beta u $ 
and the product rule we obtain
$$ \langle D^ \alpha f D^ \beta u , v \rangle =(-1)^ { |\alpha |}  \int _ { \reals ^ d} f
\sum _{ \alpha _1 + \alpha _2 = \alpha} \frac {\alpha !} { \alpha _1 ! \alpha _2 !}
D^ { \beta + \alpha _1} u D^ { \alpha _2 } v \, dx .
$$
Using the Cauchy-Schwartz inequality and the estimate \eqref{SSE1} with 
$ \lambda _1 = 0$ and $ \lambda _2=\lambda$ quickly gives \eqref{eq:trilinear2}.  The result 
for general $u$ and $v$ follows because the Schwartz class is dense in 
the space $X^ \lambda$. 

To establish \eqref{eq:map}, we write
\begin{equation} \label{qmap1}
\langle Tu, v \rangle =(-1)^{|\alpha|} 
\int_ { \reals ^d} \sum _ { \beta _1 + \beta _2 =\beta}
 \frac { \beta !}{ \beta _1! \beta _2 !} f 
 D^\alpha (D^{\beta _1 } u ( \zeta /(ih))^{ \beta_2} v ) \, dx.
\end{equation}
The estimate \eqref{eq:map} follows from \eqref{eq:trilinear2}.

In the case when $ f$ is H\"older continuous, we write $ f = f _ h + f^h$ 
where $ f_ h$ is a mollification of $f$ and the decomposition satisfies
$ \|D^ \alpha f _ h \|_ \infty \lesssim h ^ { \theta - |\alpha|}$ for
$|\alpha|\geq 1$  and $ \| f^ h \| _ \infty \lesssim h ^ \theta$. 
We use this decomposition 
of $f$ in  \eqref{qmap1}
and consider the two terms. For the first term, we use that
$ |\alpha |\geq 1$, integrate by parts and use \eqref{SSE1}
to obtain 
  $$
\left |  \int_ { \reals ^d} \sum _ { \beta _1 + \beta _2 =\beta}
  \frac { \beta !}{ \beta _1! \beta _2 !} f_ h D^ \alpha 
  ( D^ { \beta _1 } u ( \zeta /(ih))^{ \beta_2} v ) \, dx \right|
  \lesssim \|f\|_ { C^ \theta ( \reals ^d)} 
  h^ { -2 \lambda - |\alpha | - |\beta | + \theta  } \| u \|_ { X^ \lambda }
  \| v \| _ { X^ \lambda}. 
  $$
Since $ \|f^h\|_\infty \lesssim \| f\|_ { C^ \theta( \R ^d)}h^ \theta$, we may use   \eqref{eq:map} and obtain the same estimate for the term involving $f^h$, 
$$
\left |  \int_ { \reals ^d} \sum _ { \beta _1 + \beta _2 =\beta}
  \frac { \beta !}{ \beta _1! \beta _2 !} f^h D^ \alpha 
  ( D^ { \beta _1 } u ( \zeta /(ih))^{ \beta_2} v ) \, dx \right|
  \lesssim \|f\|_ { C^ \theta ( \reals ^d)} 
  h^ {  -2 \lambda - |\alpha | - |\beta | + \theta } \| u \|_ { X^ \lambda } \| v \| _ { X^ \lambda}. 
$$
\end{proof}

Our next proposition shows that multiplying by a smooth function preserves the 
class $ X^ \lambda$.

\begin{proposition}
    \label{prop:aqbound}
Let $ v \in C^ \infty ( \reals ^d)$ and assume that $v$ and  all derivatives of 
$v$ are bounded on $ \reals ^d$. Then for $ \lambda \in \reals $, we have 
    \begin{equation}\label{Mult}
      \| uv \|_{ X^ \lambda } \lesssim \| u \|_{X^ \lambda}
      .
    \end{equation}
    The implied constant in \eqref{Mult} depends on $ \lambda$ and 
    $   \| D^ \alpha v \| _{ \infty} $ for $ |\alpha | \leq c(|\lambda|)$.
  \end{proposition}

  \begin{proof} We begin by giving a proof by induction for  \eqref{Mult} when 
  $ \lambda \geq 0$ is an integer.  Since $X^0= L^2$, the base case 
  $ \lambda =0$ is elementary. 
  Next, we assume \eqref{Mult} holds for  $ \lambda \geq 0$ and establish 
  the inequality for $ \lambda +1$. Using Plancherel's theorem we have 
  $\| u v \| _ { X^ { \lambda +1}} \leq h  \| uv \|_{ X^ { \lambda }}
  + \|P(hD)(uv) \|_ { X^ \lambda}$. Then using the product rule, we have
    \begin{align*}
      \|uv\|_{X^ { \lambda+1}}&\lesssim h \| uv\|_{ X^ \lambda}
      + \| v P(hD) u \| _ { X^ \lambda }
      +h \| u Dv \| _ { X^ \lambda } + h^2 \|u \Delta v \| _ { X^\lambda}
      + h^2 \| Du \cdot Dv \|_{ X^ \lambda } \\
      & \lesssim ( h + h^2 ) \|u \| _ { X^ \lambda} + \| P(hD) u\| _ { X^ \lambda } 
      + h^2 \| Du \| _ { X^ \lambda} \\
        & \lesssim \| u \| _{X^ { \lambda +1}}.
    \end{align*}
The second inequality uses our induction hypothesis, the third inequality 
follows from \eqref{SSE1} and uses that $ 0 < h \leq 1$.

    Next, we may use complex interpolation to establish that the map
$ u \rightarrow v u $ is bounded on $ X^ \lambda$ for all $ \lambda \geq 0$. 
Finally, the adjoint of this map will also be bounded on $X^ { - \lambda}$ 
which gives \eqref{Mult} for $ \lambda <0$.
  \end{proof}
 
We now give conditions which imply that \eqref{eq:contraction} is a
contraction map. Proposition \ref{prop:Decomposition} implies that the conditions
on $q$ and $Q$ in our next theorem follow from the 
hypotheses in Theorem \ref{th:main}.  

\begin{theorem} 
\label{th:qma}
Suppose that 
\begin{equation}
\label{eq:qma}
Q_j = \sum _ { |\beta|\leq m-1} D^ \beta Q_{j\beta  },
\qquad q  = \sum _ { |\beta | \leq m}
D^ \beta q _ \beta
\end{equation}
with $ Q_ { j\beta  }, q_ \beta $ in $C^ \theta ( \R^ d)$  for some $ \theta \in (0,1]$ 
and 
each function  is supported in $ \bar \Omega$. Then there is a value $h_0$ depending on
$\theta$ and 
the H\"older norms of the functions $ Q_{j\beta}, q_\beta$ so that if 
$ 0 < h \leq h_0$, we may find a function $ \psi \in X^ { m/2}$ which satisfies
$ L _ \zeta \psi =f $ in $ \Omega$. The solution will satisfy the estimate
$$
\| \psi \| _ { X^ { m/2}} \lesssim   \|  f \| _ { X^ {- m/2}}.
$$
\end{theorem}

\begin{proof} 
If $Q$ and $q$ are as in \eqref{eq:qma}, then according to Proposition 
\ref{prop:newcon} the map \eqref{eq:contraction} will be a contraction on 
$X^ { m/2}$ at least for $h$ sufficiently small.  Thus, we may 
solve the integral equation 
\eqref{eq:intfrm} for $ \psi$ and the solution will satisfy the estimate
$
\|\psi \| _ { X^ { m/2}} \lesssim  \|I_\phi^m  f \| _ { X^ { m/2}} 
$. Then $m$ applications of  Corollary \ref{cor:Iphi} 
imply the estimate of the 
Theorem. 
To show that $\psi $ is a weak solution of the differential equation $ L_ \zeta \psi =f$
we need to show that it is a solution in the sense of distributions, which
follows from  Lemma \ref{lm:Inverse} and that $\psi$ is in $W^ { m,2 } ( \Omega)$. We use 
\eqref{SSE2} 
to show that $ X^ { m/2} \subset W^ { m,2} ( \R^d)$. 
\end{proof}


\section{Average Estimate}\label{sec:avg}
The estimate \eqref{SSE3} of Proposition \ref{prop:Df} gives an estimate for
the norm of a distribution in $X^ { -\lambda}$ in terms of the  
$W^ { -s, 2} $ norm. However,
as observed by Haberman and Tataru \cite{MR3024091}, there is an 
improvement available if we look at 
averages of $X_{\tau \zeta}^ \lambda $ norms over certain families of 
vectors $ \zeta$ and for $ \tau $ in an interval $ [ h, 2h]$.  The improvement 
is that the averages grow less rapidly in $h$. This improved behavior is used 
in section \ref{sec:Main} to recover $q$ and $Q$ 
by studying limits as $h$ tends to zero of expressions 
defined using our solutions. 
In this section, we describe the 
averages we are interested in and give an extension of Haberman and Tataru's estimate.


Fix $\xi_0 \in \R^d\setminus \{ 0 \}$ and let $\mu_1, \mu_2 \in \R^d$ such that 
$\xi_0 \cdot \mu_1 = \xi_0 \cdot \mu_2 =0$ and $  \mu_i \cdot \mu_j = \delta _{ ij}$.
We set 
\begin{equation*}
\begin{aligned}
\mu_1(\theta) &= \text{Re} (e^{i \theta} (\mu_1 + i \mu_2)) 
=  \mu_1 \cos (\theta) - \mu_2 \sin (\theta) \\
\mu_2(\theta) &= \text{Im} (e^{i \theta} (\mu_1 + i \mu_2)) 
=  \mu_1 \sin (\theta) + \mu_2 \cos (\theta)
\end{aligned}
\end{equation*}
%
and then for $\tau$ small we put
\begin{equation}
\label{eq:rotdef}
\begin{aligned}
\zeta^1(\tau,\theta) &= \mu_1(\theta) 
  + i \sqrt{1 - \frac{\tau^2 |\xi_0|^2}{4}} \mu_2(\theta) - i \frac{\tau \xi_0}{2}, \\
\zeta^2(\tau,\theta) &= - \mu_1(\theta) 
  - i \sqrt{1 - \frac{\tau^2 |\xi_0|^2}{4}} \mu_2(\theta) - i \frac{\tau \xi_0}{2}. 
\end{aligned}
\end{equation}
Throughout this section, we will let $P(\tau \xi) 
= |\tau \xi|^2 - 2 i \zeta ^k(\tau ,\theta) \cdot \tau \xi$ 
for $k = 1$ or $2$. 
We will choose $h > 0$ such that $h \leq {1} /( {4|\xi_0|})$ and 
we will have that $h \leq\tau \leq 2h$.
For $\xi \in \R^d$, we define $\xi^\perp$ to be the projection of $\xi$ 
onto the plane spanned by $\{ \mu_1 , \mu_2 \}$.
%
We will now prove an estimate for the average of the norms
\begin{equation}
    \label{IntroAverage}
 \norm{f}_{X_{\tau\zeta^k(\tau,\theta)}^{-\lambda}}, 
 \qquad h \leq \tau \leq 2h, \ \theta \in [0, 2\pi]
\end{equation}
We begin with a technical lemma which gives an estimate for  the Jacobian of 
a change of variables that will be used in studying the average of the norms 
in \eqref{IntroAverage}. 

\begin{lemma} \label{lm:Jacobian}
Let $\xi, \xi_0 \in \R^d\setminus\{ 0\} $, and 
$0 < \tau \leq \min(1,  {1}/(2 |\xi_0|))$ and let $P( \tau \xi) $ be as defined after
\eqref{eq:rotdef}. 
We define a change of variables  $ z(\tau, \theta)= (z_1(\tau, \theta), z_2 ( \tau , \theta) )$ by 
\begin{equation}  \label{eq:cov}
 z_1(\tau, \theta) = \frac{\re \, P(\tau \xi)}{\tau^2} 
 \quad \text{ and } \quad z_2(\tau,  \theta)  = \frac{\im \, P(\tau \xi)}{\tau^2} .
 \end{equation}
 Then  the Jacobian $J$ for this change of variables  has the lower bound
\[ \frac{2 |\xi^\perp|^2 }{\tau^3} \leq  J.  \]
\end{lemma}

\begin{proof}
We give the proof for  $ \zeta ^1$. The case of $ \zeta ^2$  is similar 
and we omit the details. 
A calculation shows that
\begin{align*}
J &= \Big| \text{det} \begin{pmatrix}
\displaystyle \frac { -2  \mu_2\big ( \theta)\cdot \xi  }{\tau^2} ( \sqrt { 1 - \frac { \tau^2 |\xi_0|^2} 4}
+ \frac{ \tau^2 |\xi _ 0|^2}{ 4 \sqrt { 1 - \frac { \tau^2 |\xi_0|^2} 4} } \big ) 
& {\displaystyle \frac  {2 \mu_1(\theta) \cdot \xi } \tau} \sqrt{1 - \frac{\tau^2 |\xi_0|^2}{4}} \\
2  \mu_1(\theta) \cdot \xi /\tau^2 &  2  \mu_2(\theta) \cdot \xi /\tau
\end{pmatrix} \Big| \\
&= \Big| -\frac{|\xi_0|^2}{\sqrt{1 - \frac{\tau^2 |\xi_0|^2}{4}}}
\frac{|\mu_2(\theta) \cdot \xi|^2}{\tau} 
- 4 \frac{\sqrt{1 - \frac{\tau^2 |\xi_0|^2}{4}}}{\tau^3} |\mu_2(\theta) 
\cdot \xi|^2 - 4 \frac{\sqrt{1 - \frac{\tau^2 |\xi_0|^2}{4}}}{\tau^3}
|\mu_1(\theta) \cdot \xi|^2 \Big| \\
&= \frac{4 |\xi^\perp|^2}{\tau^3} \Big( |\xi_0|^2 \tau^2 
\frac{1}{4\sqrt{1 - \frac{|\tau \xi_0|^2}{4}}} \frac{|\mu_2(\theta) 
\cdot \xi|^2}{|\xi^\perp|^2} + \sqrt{1 - \frac{\tau^2 |\xi_0|^2}{4}} \Big). 
\end{align*}
The condition that $ \tau \leq 1/( 2 |\xi_0| ) $, gives at least 
the following bound
\[   
\frac{1}{2} \leq \sqrt{1 - \frac{\tau^2 |\xi_0|^2}{4}} \leq 1 .\]
Therefore we can conclude
\[ 2 \frac{|\xi^\perp|^2 }{\tau^3} \leq  J.  \]
\end{proof}

\begin{proposition} 
\label{Gamma} 
Let $ P(\tau\xi)$ depend $\tau$ and $ \theta$ as defined 
in the beginning of this section.  
If $ 0 < \epsilon \leq 1$ and $h \langle \xi_0\rangle^2 \lesssim 1$,  
we have
$$
\frac 1 h \int _ 0 ^ { 2\pi } \int _ h ^ { 2h} 
\frac 1 { ( \tau + |P(\tau\xi ) | ) ^ { 2-2\epsilon} } \, d\tau \, d\theta
\lesssim \frac 1 { \tau^ { 4- 4 \epsilon } 
\langle  \xi \rangle ^ { 4-4\epsilon }}.
$$
\end{proposition}

\begin{proof}
We consider three  cases depending on $ \xi$. Case 1:  
$|\xi | \leq \max ( 16 |\xi _0 | , 1) $
Case 2a: $|\xi | > \max ( 16 |\xi _0 | , 1) $ and $ |\xi|^ 2 \leq 16 |\xi^ \perp| /h$,
Case 2b: $|\xi | >  \max ( 16 |\xi _0 | , 1) $ and $ |\xi|^ 2 > 16 |\xi^ \perp| /h$.

In Case 1, we use  that $ (\tau + |P(\tau \xi)|) ^ { 2 \epsilon -2} \leq 
h ^ { 2 \epsilon -2}$ and  $\langle \xi \rangle \lesssim \langle \xi_0 \rangle $ 
to obtain that 
$$
\frac 1 h \int _ 0 ^ { 2\pi } \int _h ^ { 2h } \frac 1 { ( \tau+ |P(\tau\xi ) | ) ^ { 2 - 2\epsilon}}
\, d\tau \, d\theta \lesssim \frac 1 { h ^ { 2- 2\epsilon} } 
\lesssim \frac { \langle \xi _ 0 \rangle ^ { 4- 4 \epsilon } h^ { 2- 2 \epsilon } } 
{h ^ { 4 - 4 \epsilon} \langle \xi \rangle ^ { 4- 4 \epsilon }  }
$$
which gives the desired result if we require  
that $ h^ { 2-2\epsilon} \langle \xi _0 \rangle ^ { 4-4 \epsilon } \lesssim 1$.

In Case 2a) we will use the change of variable from Lemma \ref{lm:Jacobian}. The conditions
$ |\xi |^2  \leq 16| \xi ^ \perp|/h $ and $|\xi | \geq 16 |\xi_0|$ imply that  
$|z| =| z_1 +iz_2| \leq 21 |\xi ^ \perp| /h$. Using this observation and the estimate 
for the Jacobian 
from  Lemma \ref{lm:Jacobian} gives 
\begin{multline*}
\frac 1 h  \int _ 0 ^ {2\pi } \int _ h ^ { 2h } 
\frac 1 { ( \tau + |P(\tau \xi )|)^ { 2- 2 \epsilon }}
\, d\tau \, d\theta 
\lesssim \frac { h ^ 2 } { |\xi ^ \perp|^2}
\int _{ |\tau| \leq 21 |\xi^ \perp|/h } \frac 1 { ( h ^2 |z|) ^ { 2- 2\epsilon }} \, d z \\
\lesssim 
\frac 1 { h^ { 4 - 4 \epsilon }} \left ( \frac { h  } { |\xi ^ \perp| }\right) ^ { 2- 2\epsilon}
\lesssim  \frac 1 { h ^ { 4 - 4 \epsilon } \langle \xi \rangle ^ { 4 - 4 \epsilon}}. 
\end{multline*}
The last step uses that $ |\xi |\geq 1$ and $ |\xi|^ 2 \leq 16 |\xi ^ \perp |/h $. 

Finally, in Case 2b) our conditions on $ \xi $ imply that $ |P(\tau \xi )| \geq 
h ^ 2 \langle \xi \rangle ^ 2 $ which quickly gives the conclusion of the Proposition. 
\end{proof}

We will use Proposition \ref{Gamma} to establish an estimate relating
the average of the $X^{-\lambda}$-norm of a distribution to the norm in
an  
$L^2$-Sobolev space. 
%
\begin{theorem} 
\label{th:Average}
For any $0 < \epsilon < 1$ and $2 - 2 \epsilon \leq s  \leq 2 \lambda$ we have
\[ \frac{1}{h} \int_0^{2 \pi} \int_h^{2h} \norm{f}_{X^{- \lambda}_{\zeta(\tau,\theta)}}^2 
\, d\tau \, d \theta \lesssim h^{-2(s + \lambda - 1 + \epsilon)}
\norm{f}_{{W}^{-s,2}(\reals^d)}^2 .  \]
\end{theorem}
\begin{proof}
We have
\begin{equation*}
\begin{split} 
\frac{1}{h} \int_0^{2 \pi} \int_h^{2h} \norm{f}_{X^{- \lambda}_{\zeta(\tau,\theta)}}^2 
\, d\tau \, d \theta 
&=\int _ {\R^d} |\hat f(\xi)|^2  \frac{1}{h} \int_0^{2 \pi} \int_h^{2h} 
\int_{\R^d} \frac{1}{(\tau + |P_\zeta(\tau \xi)|)^{2\lambda}} |\hat{f}(\xi)|^2   
\, d\tau \, d \theta\, d\xi  \\ 
&\lesssim \sup_\xi \frac{\langle \xi \rangle^{2s - 4 
+4\epsilon}}{(h + |P_\zeta(h \xi)|)^{ 2 \lambda - 2 + 2\epsilon}} 
\\
& \qquad \times
\int _ { \R^d}|\hat f ( \xi) |^2 \langle \xi\rangle^{ -2s + 4 - 4 \epsilon }
\frac{1}{h} \int_0^{2 \pi} \int_h^{2h}  \frac{1}{(s + |P_\zeta(\tau 
\xi)|)^{2 - 2\epsilon}}  \, d\tau \, d \theta\, d\xi   \\
& \lesssim h^ { -2 (s + \lambda -1 + \epsilon )} \|f \|^2_{\tilde W^ { -s, 2} ( \reals ^ d) } .
\end{split}
\end{equation*}
The last inequality depends on Proposition \ref{Gamma} and \eqref{eq:supremum}. 
The conditions
on $s$ are needed to apply \eqref{eq:supremum}. 
\end{proof}

\section{Cauchy data, bilinear forms and Dirichlet to Neumann maps }
\label{forms}

In this section, we establish the existence of the Cauchy data for a weak solution and 
discuss a Dirichlet to Neumann map. Our solution to the inverse boundary value 
problem in 
section \ref{sec:Main} establishes that the map from the coefficients $Q$ and $q$ to a 
bilinear form is injective. In this section, we establish that the 
bilinear form can
be determined from Cauchy data or a Dirichlet to Neumann map. Thus, our main theorem
which gives uniqueness of coefficients in terms of an hypothesis on bilinear forms 
will imply the solution of  traditional formulations of the inverse 
boundary value problem with hypotheses on Cauchy data or Dirichlet to Neumann maps.  

We  define boundary operators $\delta_{2j} u = (- \Delta)^j u | _{ \partial \Omega}$ and 
$\delta_{2j + 1} = \frac{\p}{\p \nu} (-\Delta)^j u$, both for $j =0, \dots, m-1$. By  
applying Green's identity to 
$B_0$ for $u, v \in C^\infty (\bar{\Omega})$ we obtain
\begin{equation} \label{eq:Green}
 B_0(u,v) - \int_\Omega [(-\Delta)^m u ] v = 
\sum_{j=0}^{m-1} \int_{\p \Omega} (-1)^j \delta_{j} v \delta_{2m - 1  - j}u \, d\sigma 
\end{equation}
If $u$ is a solution of an equation with principal part $ ( -\Delta )^m$, 
the vector of 
distributions on the boundary $(\delta _0 u, \dots, \delta _{ 2m-1} u)$ 
are the {\em Cauchy data }of the solution $u$. (Or at least one possible choice for the Cauchy data.) Here and throughout this section, we are  using the form  $B_0$ 
given in \eqref{eq:FormNavier}. The next Proposition
shows that we can define Cauchy data 
for a weak solution of $(-\Delta )^m u = F$. 
\begin{proposition}
\label{prop:CauchyDef}
We assume $ \partial \Omega $ is smooth. 
Suppose $u \in W^{m,2}(\Omega)$ and is a weak solution of 
$(-\Delta)^m u = F $ in  $ \Omega$ with 
$F \in \tilde{W}^{-m , 2}(\Omega)$. Then  we may define  the 
Cauchy data for $u$, 
$\delta_j u \in W^{m - \frac{1}{2} - j, 2}(\p \Omega) $ for $j = 0, \dots , 2m-1$.  
\end{proposition}

\begin{proof}
For $j = 0, \dots , m-1$ the trace theorem gives us 
$\delta_j u \in W^{m - \frac{1}{2} - j, 2}(\p
\Omega)$. Next, for 
$\phi_j \in W^{m - \frac{1}{2} - j, 2}(\p \Omega)$ 
for $j = 0, \dots, m-1$ we can solve the Dirichlet
problem
\[ \begin{cases} 
(-\Delta)^m v = 0 & \text{ in } \Omega \\
\delta_j v = \phi_j & \text{ on } \p \Omega . 
\end{cases} \]
 See, for example, the monograph of 
 Gazzola {\em et.~al.~}\cite[Theorem 2.14]{MR2667016}.
We use the identity below to define $ \delta _j u$ for $j=m, \dots, 2m-1$:  
\[ B_0(u,v) - \langle F , v \rangle 
 = \sum_{j=0}^{m-1} (-1)^j \langle  \delta_{2m - 1 - j} u , \delta _j v
 \rangle_{\partial \Omega} .\]
Here we are  using $ \langle \cdot , \cdot \rangle _{ \partial \Omega}$ to 
denote the pairing of
duality for Sobolev spaces on the boundary. 
Note that the lefthand side of this identity is defined due to our choice of the form 
$B_0$ and our assumption that $ F $ lies in the dual space to 
$W^ { m,2} ( \Omega)$, $ \tilde W ^ { -m,2} ( \Omega) $. 
\end{proof}

\begin{remark}  1. We  note that the identity \eqref{eq:Green} allows us to show that for a 
smooth 
solution, the weak definition of the Cauchy data in Proposition \ref{prop:CauchyDef} 
agrees with the classical 
definition of the operators $ \delta _j u $. 

2. 
Our  discussion of  the Cauchy data requires us to  assume that the boundary is smooth. 
For example, even to define Sobolev spaces of order $k$ on the boundary, 
we would normally need
the boundary to be $C^k$.  In addition, the existence result for
the Dirichlet problem for the polyharmonic operator that we quote 
also assumes the boundary is smooth. 

3. It is possible that the work of  Verchota  on boundary value problems
for polyharmonic operators on Lipschitz domains \cite{GV:1990} can be used to provide a formulation of
the Cauchy data for solutions  on Lipschitz domains. However, this would take us far
afield from the themes we are studying here and we have not pursued this direction. 
\end{remark}

For a weak solution of $Lu =0 $  with $L$ as defined in \eqref{eq:Poly} 
and coefficients  satisfying \eqref{eq:Qprop},  
we have that $qu$ and $Q\cdot Du$ lie in
$\tilde W ^ { -m,2} ( \Omega)$ (see \eqref{eq:trilinear}), thus 
Proposition \ref{prop:CauchyDef} implies the 
existence of the Cauchy data for $u$. 

We now consider two operators  $L_j= (-\Delta )^m u + Q^j \cdot Du +q^ju$ 
for $j =1,2$ with $Q^j$ and $q^j$ satisfying the conditions
  \eqref{eq:Qprop} and let 
$ B_j(u,v) = B_0(u,v) + \langle Q^j\cdot Du, v \rangle + \langle q^ju,v\rangle  $ 
be the form 
used to define weak solutions of the operator $L_j$. 
We say $B_1 = B_2$ if for each $u_1 \in 
W^{ m,2} (\Omega)$ a solution of $L_1 u_1 = 0$, there is a 
solution $u_2$ of $L_2 u_2 = 0$ with $u_1 - u_2 \in W^{m, 2}_0(\Omega)$ and 
$B_1(u_1,v) = B_2(u_2,v)$ for all $v \in W^{m, 2}(\Omega)$. We also assume
that we have the same result with the roles of $L_1$ and $L_2$ reversed. 

Now, suppose that $B_1 = B_2$ and $u_1$ is a solution to $L_1 u_1 = 0$ and $v_2$ a 
solution of $L_2^t v_2 = 0$. If there is a solution of $L_2 u_2 = 0$ with 
$u_1 - u_2 \in W_0^{m,2}(\Omega)$ then
\[ B_1(u_1 , v_2) = B_2(u_2 , v_2) \]
and since $v_2$ is a solution of $L_2^t v_2=0$, we have
\[ B_2(u_1 - u_2 , v_2) = 0 . \]
Thus, we can conclude
\begin{equation} 
\label{eq:equality}
B_1(u_1, v_2) = B_2(u_1 , v_2) .
\end{equation}
Next we say that two operators $L_1$ and $L_2$ have the same Cauchy 
data if 
whenever $u_1 \in W^ { m,2} ( \Omega) $ is a weak solution of  $L_1u_1=0$, 
there exists $ u_2 \in W^ { m,2} ( \Omega)$ so that   $L_2 u _2 =0$
and $\delta _j u _1 = \delta _j u _2$, $ j =0, \dots, 2m-1$. 
We observe that the condition
$ \delta _j u_1= \delta _j u _2$, $j=0, \dots, m-1$
implies $ u_1-u_2 \in W^ { m,2}_0 ( \Omega) $ and then
the definition of the Cauchy data using the form implies 
$$ B_0(u_1, v) + \langle (Q^1 \cdot D+q ^1) u_1, v \rangle 
= B_0(u_2, v) + \langle ( Q^2\cdot D + q^2) u _2, v \rangle 
$$
and thus we have  that if the Cauchy data for the two operators agree, then the forms  are equal.  This argument may be reversed so that we 
have that the 
equality of the forms is equivalent to the equality of the Cauchy data.  

Finally, in the case that we have unique solvability of a boundary 
value problem, we may view the linear space of Cauchy data as the graph of
a Dirichlet to Neumann type map. For example,  if we have unique 
solvability for the Navier 
boundary value problem \eqref{eq:navier} with data  
$( \phi_0, \phi_2, \dots, \phi_{2m-2})
\in \oplus_ { j=0} ^ { m-1} W^ { m -1/2 - 2j, 2 } ( \partial \Omega)$, then we may 
define $ \Lambda : 
 \oplus_ { j=0} ^ { m-1} W^ { m -1/2 - 2j, 2 } ( \partial \Omega) \rightarrow
  \oplus_ { j=0} ^ { m-1} W^ { m -3/2 - 2j, 2 } ( \partial \Omega)$
by 
$$ \Lambda ( \phi_0, \dots, \phi_{2m-2} ) = ( \frac \partial { \partial \nu } u, 
\frac \partial { \partial \nu } (-\Delta u) ,
\dots ,
\frac \partial { \partial \nu } (-\Delta )^ { m-1} u )
$$
where $u$ is the solution to \eqref{eq:navier} with data 
$ ( \phi_0, \dots, \phi _ { 2m-2})$.
It is clear that under the assumption of the unique solvability of the Navier boundary 
value problem, the equality of the Dirichlet to Neumann  maps  for two different operators implies
the two operators have the same Cauchy data.   One may also define Dirichlet to Neumann 
maps for other boundary value problems such as the Dirichlet problem. We leave it to
the interested reader to write down the details connecting these maps to  Cauchy data and
then the bilinear form.

We will want to work on a simply connected domain in order to recover $Q$. To do so, 
let $B$ be a ball such that $\Omega \subset \subset B$. Let 
$L_j = (-\Delta)^m + Q^j \cdot D + q^j$ in $\Omega$ and an operator 
$\tilde{L}_j$ in $B$. We define the forms of the operators as before and 
we have $\tilde{B}_j = B_j + B'$ where 
$B'$ is the form $B_0$, with domain  $ W^{m,2}(B \setminus \bar\Omega) 
\times W^{m,2}(B \setminus \bar\Omega) $.

\begin{lemma} \label{lm:Extend to Ball}
If $B_1 = B_2$ then $\tilde{B}_1 = \tilde{B}_2$.
\end{lemma}

\begin{proof}
Let $\tilde{u}_1 \in W^{m,2}(B)$ such that $L_1 
\tilde{u}_1 = 0$ and  let $u_1 =
\tilde{u}_1|_{\Omega}$. Since $B_1 = B_2$ there 
exists $u_2 \in W^{m,2}(\Omega)$ such that
$u_1 - u_2 \in W^{m,2}_0(\Omega)$ and 
$L_2 u_2 = 0$. Then, for any $v \in W^{m.2}_0(B)$
\[ \tilde{B}_1(\tilde{u}_1 , v) = B_1(u_1 , v) + B'(\tilde{u}_1 , v ) 
= B_2(u_2 , v) + B'(\tilde{u}_1 , v ) = \tilde{B}_2(\tilde{u}_2 , v) \]
where
\[ \tilde{u}_2 = \begin{cases}
\tilde{u}_1 & \text{ in } B \setminus \bar{\Omega} \\
u_2 & \text{ in } \bar{\Omega} . 
\end{cases} \]
It follows that if $\tilde{L}_1 \tilde{u}_1 = 0$ then
$\tilde{L}_2 \tilde{u}_2 = 0$. We have shown that if 
$\tilde{u}_1 \in W^{m,2}(B)$ is a solution to $\tilde{L}_1 \tilde{u}_1 = 0$ then 
there exists $\tilde{u}_2 \in W^{m, 2}(B)$ with $\tilde{L}_2 \tilde{u}_2 = 0$ 
and $\tilde{u}_1 - \tilde{u}_2 \in W^{m,2}_0(B)$ since $\tilde{u}_1 - \tilde{u}_2$ 
is compactly supported in $B$. Reversing the roles of 1 and 2, it follows 
that $\tilde{B}_1 = \tilde{B}_2$.
\end{proof}
\section{Main Result}
\label{sec:Main}
In this section, we give a construction of CGO solutions. We begin with a simple 
result that may be viewed as an analog of a theorem of P\"aiv\"arinta 
{\em et.~al.~}\cite[Theorem 1.1]{MR1993415} in the scale of the 
$X^ \lambda$ spaces. Our next step is to 
give an improvement of this result that relies on the averaging 
estimate from Section \ref{sec:avg}. We then prove our main result on the 
uniqueness of the coefficients.  
\begin{proposition}
\label{prop:simplecgo}
   Let $ \zeta \in {\mathcal V}$,    
suppose $q \in \tilde W^ { -s,p}( \Omega) $ and 
$ Q \in \tilde W ^ { 1-s, p} ( \Omega)$ with $s$ and $p$ satisfying 
$ 1/p + ( s-m)/d < 0$  and $ s < m$.
Suppose $ a \in C^ \infty ( \R^d)$ and 
$P(hD)^m a =0$. 
For $h$ small, we may find a solution  
$ \psi $ so that  $u = e ^ { x\cdot \zeta /h} (a+ \psi)$ is a 
CGO solution in $ \Omega$ and  with 
$X^ \lambda = X^\lambda _ {h \zeta}$ we have $ \psi \in X^ { m/2}$ and the estimate 
\begin{equation} 
\label{eq:cgoest1}
\|\psi \| _ { X^ { m/2}} \lesssim h^ { 2m } 
( \| Q\cdot Da \|_ { X^ { -m/2}} +h^{ -1}\|Qa\|_ { X^ { -m/2} } + \|aq \|_ { X^ {-m/2}}) . 
\end{equation}
If in addition  we have $p\geq 2$, 
then  $ \| \psi\|_ { X^ { m/2} } \lesssim h^{ 3m/2 -s}$.
\end{proposition} 
%
\begin{proof} By Sobolev embedding \cite[Theorem 6.5.1]{BL:1976}, 
we have $\tilde  W ^ { -s,p}  ( \Omega) 
\subset \tilde W ^ { - s_1, p_1}(\Omega) $ provided $ s_1 > s$,  
$p,p_1 \in ( 1, \infty)$ and 
$$ 
\frac 1 p + \frac s d  \leq \frac 1 { p_1} + \frac { s_1} { d} . 
$$
Thus we may choose $p_1 \in ( 1, \infty)$ and $ r \in(0,1)$
so that $q \in\tilde  W^ { r-m,p_1} ( \Omega) $ and 
$Q \in\tilde  W^ { 1+r-m,p_1}  ( \Omega)$ and $ 1/p_1 - r/d < 0$. By 
Proposition  \ref{prop:Decomposition} the coefficients $q$ and $Q$
have a representation as in  \eqref{eq:qma}, the hypothesis of  Theorem \ref{th:qma}. 
This follows because Morrey's Lemma implies functions in $W^ { r,p_1}( \R ^d)$ 
are H\"older 
continuous. 
Thus we may use Theorem \ref{th:qma} to solve 
$ L _ \zeta \psi = - L _ \zeta a $ and  the solution will satisfy  
$  \| \psi \| _ { X^ { m/2} } \lesssim \| L_ \zeta a \| _ { X^ { -m/2}}.$
The first estimate \eqref{eq:cgoest1}  follows  from this.
To obtain the estimate in terms of $h$, we use \eqref{SSE3} and 
Proposition \ref{prop:aqbound} to estimate the $X^  {-m/2}$ 
norm in terms of the Sobolev norm. 
\end{proof} 
The next Lemma is rather technical but important. It provides a sequence of 
CGO solutions that
are instrumental in the proof of the main theorem. 
\begin{lemma} \label{lm:psi bound}
Suppose  $ 1/p +(s- m)/d <0$, $p \geq 2$, and $s < \frac{m}{2} + 1 $. 
Assume we have potentials $ q^k$, $Q^k$ for $k=1,2$ which satisfy  
$q^k \in \tilde W ^ { -s,p}( \Omega)$ and 
$Q^k \in  \tilde{W}^{-s + 1 ,p}(\Omega)$.  
Assume $a^k   \in C^ \infty ( \R^d)$ and $P(hD)^m a^k=0$. 
Given $ \xi_0 \in \reals ^d \setminus \{ 0 \} $ and 
unit vectors $ \mu_1$ and $\mu_2$  so that $\{ \xi _0, \mu _1 , \mu_2\}$ are mutually 
orthogonal, there exist sequences 
$\{ h_j\}$, $ \{\theta _j\}$   with 
$ \lim _ {j \rightarrow \infty } h _ j =0$ and
$ \lim _{ j \rightarrow\infty } \theta _j = \theta_0$ 
with the 
following properties. 
If  $ \zeta^k _j = \zeta^k( h_j , \theta_j)$  is   
as in \eqref{eq:rotdef} and $ X_{k,j}^ \lambda = X^ \lambda _ { h_j \zeta ^k_j}$, then 
we  may find    sequences 
$\{ \psi_j^k\} \subset   X^ { m /2} _ {k,j} $, $k=1,2$ so that
$e^ { x\cdot \zeta^k_j/h_j} ( a^k + \psi^k_j) $ is a  CGO solution to 
$ L_k u=0$ and
\begin{equation}
\label{eq:avq}
 \| q^\ell \| _ { X^ { -m/2} _{k,j} } + h_j^{-1} \| Q^\ell\|_{X_{k,j}^ { -m/2}} 
+h_j^ { -3/2}\| Q ^ \ell \| _{ X_{k,j}^ { (1-m)/2}  } \lesssim
h_j  ^ { -m/2 -s +1- \epsilon }, \qquad k,\ell = 1, 2 
\end{equation}
\begin{equation} \label{eq:psi bound}
    \norm{\psi_j^k}_{X_{k,j}^{m/2}} \lesssim h_j^{ \frac{3m}{2} -s + 1 - \epsilon  } .
\end{equation}
\end{lemma}
%
%
\begin{proof}
We may use  Proposition \ref{prop:simplecgo}  to find solutions  $ \psi ^k$ 
for each $h $ and $ \theta$ which lie in the space 
$ X_k^ { m/2}$ where we use 
$X_k^ \lambda = X^ { \lambda} _{ h \zeta^k(h,\theta)}$. These solutions $ \psi^k$ 
 satisfy the
estimate \eqref{eq:cgoest1}. 
Using Proposition \ref{prop:aqbound}, we can 
bound the right hand side of
\eqref{eq:cgoest1} by $ h^{2m}(\|q^k\|_ { X_k^ { -m/2}} + h^ { -1} \| Q^k\|_ { X_k^{ -m/2}})$. 
If we sum the estimates of Theorem \ref{th:Average} for $k=1,2$, may find  
sequences  $\{h _j\} $ and $\{ \theta_j\} $  so that for $k,\ell=1,2$ we have the estimates in the 
$X^ { -m/2}_{k,j}$ norm in 
\eqref{eq:avq}. The estimate for the remaining term then follows from \eqref{SSE1}. 
After perhaps passing to a sub-sequence, 
we have that the sequence $ \theta _j$ will be convergent and we let $ \theta_0$ be the limit 
of this sequence. 
The estimate \eqref{eq:psi bound} follows from \eqref{eq:avq} and \eqref{eq:cgoest1}. 
\end{proof}
We are now ready to give the proof of our result establishing that the bilinear 
form for the operator $L$ uniquely determines the coefficients $q$ and $Q$. 
\begin{theorem}\label{th:Unique}
Suppose that we have two operators $ L_k$, $k=1,2$  as in \eqref{eq:Poly} and we have
$s < {m}/{2} + 1$ and $p\geq 2$ which satisfy  $ 1/p + (s-m) /d <0$, $p \geq 2$.
Suppose  the coefficients of the operators $L_k$ satisfy 
$Q^k \in \tilde{W}^{-s + 1 , p}(\Omega)$ 
and $q^k \in \tilde{W}^{-s , p}(\Omega)$. 
If the bilinear forms for the operators satisfy  $B_1=B_2$, 
then   $Q^1 = Q^2$ and $q^1 = q^2$.
\end{theorem}
\begin{proof}
We begin by
 applying Lemma \ref{lm:Extend to Ball} to assume 
 $ \Omega $ is a ball.
Since the Sobolev spaces $\tilde W ^ {- s,2} ( \Omega)  $ 
are increasing as $-s$ decreases, we
may assume that  $s > m/2$ and choose $\epsilon $ so that 
 $s = \frac{m}{2} + 1 - \frac{3}{2} \epsilon$ for 
 $0 < \epsilon < \frac{1}{3}$.  
Let 
 $\xi_0 \in \R^d$ and  find unit vectors 
 $\mu_1 , \mu_2 \in \R^d$ such that $\{ \xi_0 , \mu_1 , \mu_2 \}$
 are mutually orthogonal. Given  amplitudes $ a^k$ which satisfy $P(hD)^ma^k=0$. 
  We 
 apply  Lemma \ref{lm:psi bound} to the operators $L_1$ and
 $L _2^t$ to find a sequence of solutions  $ \{ \psi _j^k \} $,   a  
 sequence of spaces  
 $\{X_{k,j} ^ { m/2}\}$  and numerical sequences $\{ h_j\}$ and $ \{ \theta _j \}$ 
 so that $u_k=e^ { \zeta^k_j\cdot x /h_j} (a^k + \psi ^k _ j )$ are  CGO 
 solutions of 
 $L_1u_1 =0$ and $L_2^t u _2=0 $ which satisfy the estimates \eqref{eq:avq} 
 and \eqref{eq:psi bound}.
%
%
It is clear from the definitions of $ \zeta^k$ in \eqref{eq:zetadef} that 
$ \zeta_j ^1 + \zeta_j ^2 = - i h_j \xi_0$. 
Since we have that the forms for $L_1$ and $L_2$ are equal (see \eqref{eq:equality}), we have 
\begin{align*}
    0 &= i h_j [ B_1(u_1 , u_2) - B_2(u_1 , u_2) ] \\
    &={\zeta_j^1} \cdot \langle (Q^1 - Q^2)a^1 ,a^2 e^{-i x \cdot \xi_0} \rangle 
    + {\zeta_j^1}\cdot\langle (Q^1 - Q^2) \psi_j^1 ,a^2 e^{-i x \cdot \xi_0}  \rangle
    +{\zeta_j^1} \cdot \langle (Q^1 - Q^2) a^1, e^{-i x \cdot \xi_0} \psi_j^2 \rangle
   \\
   &\quad + {\zeta^1_j} \cdot \langle (Q^1 - Q^2) \psi_j^1 , e^{-i x \cdot \xi_0} \psi_j^2 \rangle 
      + ih_j\langle (Q^1 - Q^2) \cdot D a^1, a^2 e^{-i x \cdot \xi_0} \rangle 
      \\
      &\quad + ih_j\langle (Q^1 - Q^2) \cdot D a^1,  \psi^2_je^{-i x \cdot \xi_0} \rangle 
   + ih_j\langle (Q^1 - Q^2) \cdot D \psi_j^1 , a^2e^{-i x \cdot \xi_0} \rangle 
   \\
    &\quad  + ih_j\langle (Q^1 - Q^2) \cdot D \psi_j^1 , e^{-i x \cdot \xi_0} \psi_2 \rangle 
    + ih_j\langle (q^1 - q^2)(a^1+\psi_j^1) , e^{-i x \cdot \xi_0} (a^2 + \psi_j^2) \rangle \\
    &= I + II + III + IV + V + VI + VII +VIII + IX.  
\end{align*}
Since the only dependence on $h$ in $I$ is through the formula for 
$ \zeta^1_j$ (see \eqref{eq:zetadef}), we have 
\[ \lim _ { j \rightarrow \infty } I = \lim_{j \to \infty} {\zeta_j^1} \cdot \langle (Q^1 - Q^2)a^1 ,a^2 e^{-i x \cdot \xi_0} \rangle = -ie^{ i \theta_0}
\langle (\mu_1 + i \mu_2) \cdot  (Q^1- Q^2 ) , a^1 a^2 e^ { -ix\cdot \xi _0} \rangle
\]
For $II$, we use the duality of $X^{m/2}_{1,j} $  and $X^ { -m/2}_{1,j}$, 
the estimates \eqref{Mult},  \eqref{eq:avq} and \eqref{eq:psi bound} to obtain
\[ |II| = | {\zeta_j^1}\cdot\langle (Q^1 - Q^2) \psi_j^1 ,a^2 e^{-i x \cdot \xi_0}  \rangle| \lesssim \norm{(Q^1 - Q^2)}_{X_{1,j}^{-m/2}} \norm{\psi_j^1}_{X_{1,j}^{m/2}} 
\lesssim h_j^{m -2s + 2 - 2 \epsilon } = h_j^ \epsilon  .  \]
where the last equality is our definition of $\epsilon$.
For $III$, we use the same argument with $X_{1,j}^\lambda$ replaced 
by $X_{2,j}^ \lambda$
\[ |III| = |{\zeta_j^1} \cdot \langle (Q^1 - Q^2) a^1, e^{-i x \cdot \xi_0} \psi_j^2 \rangle| 
\lesssim \norm{(Q^1 - Q^2) }_{X_{2,j}^{-m/2}} \norm{\psi_j^2}_{X^{m/2}_{2,j}} 
\lesssim h_j^{m -2s +2  - 2 \epsilon  }  = h_j ^ \epsilon . \]
For $IV$, we  write   
$ Q^k$ in the form \eqref{eq:qma}, use the estimate \eqref{eq:trilinear2}  
and the estimates \eqref{eq:psi bound} for $ \psi_j^k$
to conclude 
\[
 |IV| = | {\zeta^1_j} \cdot \langle (Q^1 - Q^2) \psi_j^1 , e^{-i x \cdot \xi_0} \psi_j^2 \rangle| \lesssim h_j^{1-2m}\|\psi^1_ j\| _ { X_{ 1,j}^{ m/2} } 
\| \psi _j^2 \|_{X_{2,j}^{m/2}} \lesssim h_j^ { m+1 -2s + 2 - 2 \epsilon }
\lesssim h_j^{1+ \epsilon }.  
\]
%
Since the amplitudes are smooth and there is no dependence on
$\psi_k^j$, we have $|V| = |ih_j\langle (Q^1 - Q^2) \cdot D a^1, a^2 e^{-i x \cdot \xi_0} \rangle | \lesssim h_j$.
For $VI$, we use duality in $X^ { (m-1)/2}_{1,j}$, \eqref{SSE1} and \eqref{eq:psi bound}  to obtain  
\begin{multline*}
|VI| = |ih_j\langle (Q^1 - Q^2) \cdot D a^1,  \psi^2_je^{-i x \cdot \xi_0} \rangle |
\\
\lesssim h_j \norm{Q^1 - Q^2 }_{X_{1,j}^{(1-m)/2 }} 
\norm{D \psi_1}_{X_{1,j}^{(m-1)/2 }}   
\lesssim h_j^{m+1 -2s + 2 -2 \epsilon } = h_j^{1 + \epsilon} 
\end{multline*}
by  $\eqref{eq:psi bound}$ and \eqref{SSE1} 
 We may estimate $VII$ using the
same argument we used for $III$ giving 
$$
|VII| = |ih_j\langle (Q^1 - Q^2) \cdot D \psi_j^1 , a^2e^{-i x \cdot \xi_0} \rangle| \lesssim h_j . 
$$
For $VIII$,  we use the representation as in \eqref{eq:qma}  and 
\eqref{eq:trilinear2} to obtain 
\[ |VIII| = |ih_j\langle (Q^1 - Q^2) \cdot D \psi_j^1 , e^{-i x \cdot \xi_0} \psi_2 \rangle|
\lesssim h_j h_j^ { -2m} \| \psi_j^1\|_{ X^ { m/2}_{1,j}} 
\| \psi _j^2\|_{X_{2,j}^{ m/2}}
\lesssim h_j h_j^ { m -2s + 2- \epsilon } = h_j^ { 1+ \epsilon} . 
\]
where we have also used \eqref{eq:psi bound}. 
For $IX$, we expand the product $(a^1+ \psi _j^1 ) ( a^2+ \psi^2_j)$ 
and use similar arguments to obtain
\begin{align*}
|IX| &= |ih_j\langle (q^1 - q^2)(a^1+\psi_j^1) , e^{-i x \cdot \xi_0} (a^2 + \psi_j^2) \rangle| \\ &\lesssim h_j (| \hat{q}^1(\xi_0) - \hat{q}^2(\xi_0) |
      + \norm{q^1 - q^2}_{X_{2,j}^{-m/2}} \norm{\psi_j^2}_{X_{2,j}^{m/2}} 
      + \norm{q^1 - q^2}_{X_{1,j}^{-m/2}} \norm{\psi_1}_{X_{1,j}^{m/2}} \\
    &+ h^ { -2m} \norm{\psi_1}_{X_{1,j}^{m/2}} \norm{\psi_{1,j}}_{X^{m/2}} ) \\
    & \lesssim h_j (1 + h_j^{m  -2s+ 2 -2 \epsilon } ) = h_j + h_j^ { 1+ \epsilon}  . 
\end{align*}
%
Thus we can conclude 
\begin{equation}
\label{eq:identity}
0=\lim _ {j \rightarrow \infty } i e^ { -i \theta_0}h_j ( B(u_1,u_2)-B_2(u_1, u_2) =
\langle (\mu_1 + i \mu_2) \cdot (Q^1- Q^2) , a^1 a^2 e^ { -ix\cdot \xi _0} \rangle
\end{equation}
and if we set $ a^1=a^2=1$, then we conclude
\[ (\mu_1 + i \mu_2)  \cdot (\hat Q^1( \xi_0) -  \hat Q^2( \xi_0 ) )=0  .  \]
If we repeat this argument with   $ -\mu_2$ replacing $\mu_2$ we obtain
\[ (\mu_1 - i\mu_2)\cdot  (\hat Q^1( \xi_0) -  \hat Q^2( \xi_0 ) )=0  .  \]
Adding the last two displays 
we have  $ \mu\cdot ( \hat Q^1 ( \xi_0) - \hat Q^ 2( \xi_0) ) =0$ if 
$\mu$ is perpendicular 
to $\xi_0$. Choosing $ \mu $ 
parallel to 
$e_k \xi _{ 0, j} - e_j \xi _{ 0,k}$ we conclude that 
\[ 
\frac \partial { \partial x_k } ( Q^1_j-Q^2_j) - \frac \partial { \partial x_j} ( Q^1_k -Q^2_k) = 0 .
\]
This gives us that $\curl (Q^1 - Q^2) = 0$. Since we assume $ \Omega$ is a ball, 
we may use  
Proposition \ref{prop:Poincare Lemma} to find $g$ which is supported in 
$ \bar \Omega$ and so that 
 $Q^1 - Q^2 = Dg$.  Now, we use \eqref{eq:identity}  with $a^1=1$ and 
 $a^2= ( \mu_1 - i \mu_2) \cdot x/2$. This choice for $ a^2$ gives that $(\mu_1+ i \mu_2) \cdot D a^2 = -i$. 
Using  that $(Q^1 - Q^2) = Dg$, we send $h_j \to 0$ and integrate by parts to obtain that 
\[ 
0 = (\mu_1+i\mu_2) \cdot \langle Dg , e^{- i x \cdot \xi_0} a_2 \rangle 
  = i\langle g , e^{- i x \cdot \xi_0} \rangle 
  \]
where we used  that $(\mu_1 + i \mu_2 )\cdot \xi_0 = 0$. 
This gives us that $g = 0$ and we have that $Q^1 = Q^2$ as desired.
Now that we have  $Q^1 = Q^2$, we turn to the proof that    $q^1 = q^2$.  
We use our CGO solutions 
with $ a^1=a^2=1$ 
  to obtain
\[ 0 = B_1(u_1 , u_2) - B_2(u_1 , u_2) = \langle (q^1 - q^2) (1 + \psi_1) , 
e^{-i x \cdot \xi_0} (1 + \psi_2) \rangle \]
Arguing as above (see the argument for the terms $I$-$IV$), we have 
\[ \lim _ { j \rightarrow \infty} \langle (q^1-q^2 ) ( 1 +\psi_1^j), (1+ \psi_2^j)e^ { -i x\cdot \xi _0}\rangle = \hat q^1( \xi _0 ) - \hat q^2( \xi_0 ). \]
Since $ \hat q^1 - \hat q^2 =0$, we have $q^1 = q^2$. 
\end{proof}
\begin{proof}[Proof of Main Theorem, Theorem \ref{th:main}] 
In section \ref{forms}, we show that if the Dirichlet to Neumann to maps for two
operators are equal, then  the two forms are equal. Thus this result follows immediately 
from Theorem \ref{th:Unique}. 
\end{proof} 
\appendix
\section{Sobolev spaces}
\renewcommand{\theequation}{\Alph{section}.\arabic{equation}}
%
%
We establish several  representation theorems for   elements of Sobolev spaces that are 
used 
in our arguments. 
Our results will depend on techniques developed by 
D.~Mitrea, M.~Mitrea, and S.~Monniaux \cite{MR2425010} who give a 
version of the Poincar\'e Lemma in their  study of the Poisson problem for the
exterior derivative.  They construct a homotopy for the de Rham complex 
in a Lipschitz domain that is star-shaped with respect to each point in a 
small ball. The homotopy operators have the useful property that they 
preserve the class of forms that are supported in the domain. 
We briefly summarize what we need from 
their work and refer the reader to their article for precise definitions and 
complete statements of their results. 
They work with differential forms and initially assume the coefficients lie
in $C_0^\infty  (\Omega)$. Thus, let $ \Lambda $ denote the 
exterior algebra on $ \R^d$ and let
$\Lambda ^j$ be the elements of degree $j$. We will use 
$ C_0^ \infty(\Omega; \Lambda ^j)$ 
to denote the differential forms of degree $j$ on $ \Omega$ 
with coefficients
in $C_0 ^ \infty( \Omega)$. In \cite[Theorem 4.1]{MR2425010}, 
Mitrea  {\em et.~al.~}
find a family of operators 
 $J_j : C_0^ \infty (\Omega; \Lambda^j) \rightarrow 
 C_0^ \infty(\Omega ; \Lambda ^{ j-1})$, 
 for $ j=1, \dots, d$ 
 with the property that if $\omega \in C_0^ \infty( \Omega ; \Lambda ^j)$
\begin{equation}\label{MMMhomotopy}
\omega = \left\{ \begin{aligned} & J_1 ( d \omega) , \qquad &&j  =0 \\
    & J_{j+1} ( d \omega ) + d J_j (\omega) , \qquad && j =1, \dots, d-1\\
      &     d J_d( \omega) + \langle  \omega ,\theta   V_d \rangle \theta V_d  , \qquad 
        && j = d  
\end{aligned} \right. 
\end{equation}
Here, we are  temporarily using
$d$ to denote both the dimension and the exterior derivative on $ \R^d$. 
In the third line, $ \theta \in C_0^ \infty ( \Omega) $ is determined in the
construction of the operator $J_d$,  $V_d= dx_1\wedge \dots \wedge dx_d$ is the 
standard volume form on $ \R^d$ 
and
the pairing of duality $\langle \cdot , \cdot \rangle $ is defined in section 2 of 
Mitrea {\em et.~al.~}\cite{MR2425010}. 
In addition, the operators  $J_j$ map 
\begin{equation}
J_j:\tilde W^ { s,p}( \Omega;\Lambda ^j ) \rightarrow \tilde 
W^{s+1,p}( \Omega; \Lambda^{ j-1}) .
\end{equation}
To see this, consider $J_j$ acting on a larger domain, and use the estimate of 
Theorem 3.8 in Mitrea {\em et.~al.}~and the property that $J_j$ preserves the 
property 
of being supported in $\Omega$. 
\begin{proposition} 
\label{prop:Decomposition}
Assume that $ \Omega $ is a Lipschitz domain. 
Suppose that $ f \in \tilde W^ { -s,p}( \Omega) $ with $ s= k-r$ where 
$ 0 \leq r < 1$ and $k\geq 1$ is an integer. Then we may write 
$$ f = \sum _ { |\alpha | = k } D^ \alpha F _\alpha  +F_0$$
with $ F_\alpha \in \tilde{W}^ { r,p}( \Omega)$ and $ F_0 \in C_c^ \infty ( \Omega)$
\end{proposition}  
\begin{proof}
Using a partition of unity, we may reduce to the case where the domain $ \Omega$ is 
star-shaped with 
respect to each point in a ball. We fix $f \in \tilde W ^ {s,p} ( \Omega)$ and 
apply the case $j=d$ of
\eqref{MMMhomotopy} to the form $ \omega = f V_d$. If we suppress the differentials, this can be
rewritten as 
$$ f = D \cdot F + \langle u , \theta\rangle \theta 
$$
with $F = ( F_1, \dots, F_d)$ and $F_j \in \tilde W ^ { s+1, p}( \Omega)$. If $ k=1$, we are 
done 
and for $ k \geq 2$, we repeat this 
argument to represent each $F_j$ as a smooth term and derivatives of functions 
in $ \tilde W^ { s+2,p} ( \Omega)$. 
After $k$ steps, we obtain the desired representation. 
\end{proof}
We remark that the specific case of Mitrea's results used in the proof of
Proposition \ref{prop:Decomposition} may be found in earlier work 
of Bogovskii \cite{MR631691}.
\begin{proposition} \label{prop:Poincare Lemma}
Suppose that $ \Omega $ is a ball in $\R^d$ and \(Q = ( Q_1, \dots, Q_d)\) lies 
$\tilde W^ { s,p} ( \R^d)$ with $ \curl Q = 0$, then we may find 
$g \in \tilde W ^ { s+1, p} ( \Omega) $ so that $ Q = Dg$. 
\end{proposition}
\begin{proof}
We form a differential form $ \omega = \sum _ { k =1} ^ d Q _k dx_k$.  Our hypothesis $\curl Q = 0 $ 
gives immediately that $ d \omega =0$ and then the Proposition follows from the case $j=1$ of 
\eqref{MMMhomotopy}. 
\end{proof}
\bibliographystyle{plain}

\begin{thebibliography}{10}

\bibitem{MR4027047}
Yernat Assylbekov and Karthik Iyer.
\newblock Determining rough first order perturbations of the polyharmonic
  operator.
\newblock {\em Inverse Probl. Imaging}, 13(5):1045--1066, 2019.

\bibitem{MR3627033}
Yernat~M. Assylbekov.
\newblock Inverse problems for the perturbed polyharmonic operator with
  coefficients in {S}obolev spaces with non-positive order.
\newblock {\em Inverse Problems}, 32(10):105009, 22, 2016.

\bibitem{MR3691848}
Yernat~M. Assylbekov.
\newblock Corrigendum: {I}nverse problems for the perturbed polyharmonic
  operator with coefficients in {S}obolev spaces with non-positive order (2016
  { {I}nverse {P}roblems} 32 105009) [ {MR}3627033].
\newblock {\em Inverse Problems}, 33(9):099501, 2, 2017.

\bibitem{AP:2006}
K.~Astala and L.~P{\"a}iv{\"a}rinta.
\newblock Calder\'on's inverse conductivity problem in the plane.
\newblock {\em Ann. of Math. (2)}, 163(1):265--299, 2006.

\bibitem{BL:1976}
J.~Bergh and J.~L{\" o}fstr{\"o}m.
\newblock {\em Interpolation spaces}.
\newblock Springer-Verlag, Berlin, New York, 1976.

\bibitem{MR3953481}
Sombuddha Bhattacharyya and Tuhin Ghosh.
\newblock Inverse boundary value problem of determining up to a second order
  tensor appear in the lower order perturbation of a polyharmonic operator.
\newblock {\em J. Fourier Anal. Appl.}, 25(3):661--683, 2019.

\bibitem{BG:2021}
Sombuddha Bhattacharyya and Tuhin Ghosh.
\newblock An inverse problem on determining second order symmetric tensor for
  perturbed biharmonic operator.
\newblock {\em Math. Annalen}, 2021.
\newblock Published online, 16 October 2021.

\bibitem{MR631691}
M.~E. Bogovski{\u\i}.
\newblock Solutions of some problems of vector analysis, associated with the
  operators {${\rm div}$} and {${\rm grad}$}.
\newblock In {\em Theory of cubature formulas and the application of functional
  analysis to problems of mathematical physics}, volume 1980 of {\em Trudy Sem.
  S. L. Soboleva, No. 1}, pages 5--40, 149. Akad. Nauk SSSR Sibirsk. Otdel.
  Inst. Mat., Novosibirsk, 1980.

\bibitem{MR1209299}
J.~Bourgain.
\newblock Fourier transform restriction phenomena for certain lattice subsets
  and applications to nonlinear evolution equations. {I}. {S}chr\"{o}dinger
  equations.
\newblock {\em Geom. Funct. Anal.}, 3(2):107--156, 1993.

\bibitem{MR1215780}
J.~Bourgain.
\newblock Fourier transform restriction phenomena for certain lattice subsets
  and applications to nonlinear evolution equations. {II}. {T}he
  {K}d{V}-equation.
\newblock {\em Geom. Funct. Anal.}, 3(3):209--262, 1993.

\bibitem{RB:1994a}
R.M. Brown.
\newblock Global uniqueness in the impedance imaging problem for less regular
  conductivities.
\newblock {\em SIAM J. Math. Anal.}, 27:1049--1056, 1996.

\bibitem{BT:2003}
R.M. Brown and R.H. Torres.
\newblock Uniqueness in the inverse conductivity problem for conductivities
  with {$3/2$} derivatives in {$L\sp p,\ p>2n$}.
\newblock {\em J. Fourier Anal. Appl.}, 9(6):563--574, 2003.

\bibitem{AC:1980}
A.~P. Calder\'on.
\newblock On an inverse boundary value problem.
\newblock In {\em Seminar on Numerical Analysis and its Applications to
  Continuum Physics}, pages 65--73, Rio de Janeiro, 1980. Soc. Brasileira de
  Matem\'atica.

\bibitem{MR3456182}
Pedro Caro and Keith~M. Rogers.
\newblock Global uniqueness for the {C}alder\'{o}n problem with {L}ipschitz
  conductivities.
\newblock {\em Forum Math. Pi}, 4:e2, 28, 2016.

\bibitem{MR3357587}
Anupam~Pal Choudhury and Venkateswaran~P. Krishnan.
\newblock Stability estimates for the inverse boundary value problem for the
  biharmonic operator with bounded potentials.
\newblock {\em J. Math. Anal. Appl.}, 431(1):300--316, 2015.

\bibitem{MR2667016}
Filippo Gazzola, Hans-Christoph Grunau, and Guido Sweers.
\newblock {\em Polyharmonic boundary value problems}, volume 1991 of {\em
  Lecture Notes in Mathematics}.
\newblock Springer-Verlag, Berlin, 2010.
\newblock Positivity preserving and nonlinear higher order elliptic equations
  in bounded domains.

\bibitem{MR3546596}
Tuhin Ghosh and Venkateswaran~P. Krishnan.
\newblock Determination of lower order perturbations of the polyharmonic
  operator from partial boundary data.
\newblock {\em Appl. Anal.}, 95(11):2444--2463, 2016.

\bibitem{MR3024091}
B.~Haberman and D.~Tataru.
\newblock Uniqueness in {C}alder\'on's problem with {L}ipschitz conductivities.
\newblock {\em Duke Math. J.}, 162(3):496--516, 2013.

\bibitem{MR3397029}
Boaz Haberman.
\newblock Uniqueness in {C}alder\'{o}n's problem for conductivities with
  unbounded gradient.
\newblock {\em Comm. Math. Phys.}, 340(2):639--659, 2015.

\bibitem{MR4273826}
Seheon Ham, Yehyun Kwon, and Sanghyuk Lee.
\newblock Uniqueness in the {C}alder\'{o}n problem and bilinear restriction
  estimates.
\newblock {\em J. Funct. Anal.}, 281(8):Paper No. 109119, 58, 2021.

\bibitem{MR2873860}
Katsiaryna Krupchyk, Matti Lassas, and Gunther Uhlmann.
\newblock Determining a first order perturbation of the biharmonic operator by
  partial boundary measurements.
\newblock {\em J. Funct. Anal.}, 262(4):1781--1801, 2012.

\bibitem{MR3118392}
Katsiaryna Krupchyk, Matti Lassas, and Gunther Uhlmann.
\newblock Inverse boundary value problems for the perturbed polyharmonic
  operator.
\newblock {\em Trans. Amer. Math. Soc.}, 366(1):95--112, 2014.

\bibitem{MR3484382}
Katsiaryna Krupchyk and Gunther Uhlmann.
\newblock Inverse boundary problems for polyharmonic operators with unbounded
  potentials.
\newblock {\em J. Spectr. Theory}, 6(1):145--183, 2016.

\bibitem{MR1742312}
William McLean.
\newblock {\em Strongly elliptic systems and boundary integral equations}.
\newblock Cambridge University Press, Cambridge, 2000.

\bibitem{MR2425010}
D.~Mitrea, M.~Mitrea, and S.~Monniaux.
\newblock The {P}oisson problem for the exterior derivative operator with
  {D}irichlet boundary condition in nonsmooth domains.
\newblock {\em Commun. Pure Appl. Anal.}, 7(6):1295--1333, 2008.

\bibitem{MR1993415}
Lassi P\"{a}iv\"{a}rinta, Alexander Panchenko, and Gunther Uhlmann.
\newblock Complex geometrical optics solutions for {L}ipschitz conductivities.
\newblock {\em Rev. Mat. Iberoamericana}, 19(1):57--72, 2003.

\bibitem{poncevanegas2019bilinear}
Felipe Ponce-Vanegas.
\newblock The bilinear strategy for {C}alder\'on's problem, 2019.
\newblock arXiv:1908.04050 [math.AP].

\bibitem{MR4112132}
Felipe Ponce-Vanegas.
\newblock Reconstruction of the derivative of the conductivity at the boundary.
\newblock {\em Inverse Probl. Imaging}, 14(4):701--718, 2020.

\bibitem{MR98a:47071}
T.~Runst and W.~Sickel.
\newblock {\em Sobolev spaces of fractional order, {N}emytskij operators, and
  nonlinear partial differential equations}.
\newblock Walter de Gruyter \& Co., Berlin, 1996.

\bibitem{MR3476498}
V.~S. Serov.
\newblock Borg-{L}evinson theorem for perturbations of the bi-harmonic
  operator.
\newblock {\em Inverse Problems}, 32(4):045002, 19, 2016.

\bibitem{SU:1987}
J.~Sylvester and G.~Uhlmann.
\newblock A global uniqueness theorem for an inverse boundary value problem.
\newblock {\em Annals of Math.}, 125:153--169, 1987.

\bibitem{GV:1984}
G.C. Verchota.
\newblock Layer potentials and regularity for the {D}irichlet problem for
  {L}aplace's equation on {L}ipschitz domains.
\newblock {\em J. Funct. Anal.}, 59:572--611, 1984.

\bibitem{GV:1990}
Gregory Verchota.
\newblock The {D}irichlet problem for the polyharmonic equation in {L}ipschitz
  domains.
\newblock {\em Indiana Univ. Math. J.}, 39(3):671--702, 1990.

\bibitem{LY:2020}
Lili Yan.
\newblock Inverse boundary problems for biharmonic operators in transversally
  anisotropic geometries.
\newblock arXiv:2012.14273 [math.AP].

\bibitem{LY:2021}
Lili Yan.
\newblock Reconstructing a potential perturbation of the biharmonic operator on
  transversally anisotropic manifolds.
\newblock arXiv:2109.07712.

\end{thebibliography}
%

  \def\cprime{$'$} \def\cprime{$'$}

\medskip
\small
\noindent
posted to arXiv on 26 August 2021, \\ revised 7 December 2021

\end{document}